\documentclass[11pt,letterpaper]{amsart}
\usepackage{hyperref}
\usepackage{stmaryrd}
\usepackage{amsmath, amssymb, amsthm}
\usepackage[cmtip,arrow]{xy} 
\usepackage{enumerate}
\usepackage{pb-diagram,pb-xy} 
\usepackage{mathrsfs}

\newcommand{\Q}{\mathbb{Q}}

\newcommand{\Z}{\mathbb{Z}}
\newcommand{\C}{\mathbb{C}}

\newcommand{\A}{\mathbb{A}}
\renewcommand{\P}{\mathbb{P}}

\usepackage{fullpage}

\theoremstyle{plain}
\newtheorem{theorem}{Theorem}[section]
\newtheorem*{introtheorem}{Theorem}

\newtheorem{proposition}[theorem]{Proposition}

\newtheorem{corollary}[theorem]{Corollary}

\theoremstyle{remark}
\newtheorem{remark}[theorem]{Remark}

\theoremstyle{definition}
\newtheorem{definition}[theorem]{Definition}

\newtheorem{example}[theorem]{Example}

\usepackage{graphicx}
\usepackage[cmtip,arrow]{xy} 
\usepackage{enumerate}
\usepackage{verbatim}

\theoremstyle{definition}

\newtheorem*{Local Vanishing}{\mdseries{\emph{Local Vanishing}}}
\newtheorem*{Skoda's Theorem}{\itshape{Skoda's Theorem}}

\newcommand{\J}{\mathcal{J}}
\renewcommand{\a}{\mathfrak{a}}
\newcommand{\m}{\mathfrak{m}}
\renewcommand{\O}{\mathcal{O}}

\newcommand{\K}{K_\pi}
\DeclareMathOperator{\ord}{ord}

\DeclareMathOperator{\spec}{Spec}
\DeclareMathOperator{\Frac}{Frac}

\begin{document}
\relax
\bibliographystyle{amsalpha}

\title[Jumping Numbers on Algebraic Surfaces]{Jumping Numbers on Algebraic Surfaces with Rational Singularities}
\author{Kevin Tucker}
\address{Department of Mathematics, University of Michigan, Ann Arbor, Michigan  48109}
\email{kevtuck@umich.edu}
\thanks{The author was partially supported by the NSF under grant DMS-0502170.}
\dedicatory{In memory of Juha Heinonen}

\begin{abstract}
In this article, we study the jumping numbers of an ideal in the local ring at rational singularity on a complex algebraic surface.  By understanding the contributions of reduced divisors on a fixed resolution, we are able to present an algorithm for finding of the jumping numbers of the ideal.  This shows, in particular, how to compute the jumping numbers of a plane curve from the numerical data of its minimal resolution.  In addition, the jumping numbers of the maximal ideal at the singular point in a Du Val or toric surface singularity are computed, and applications to the smooth case are explored.
\end{abstract}

\maketitle
\tableofcontents

\section{Introduction}

To every sheaf of ideals $\a$ on a complex algebraic variety $X$ with mild singularities, one can associate its  multiplier ideals $\J(X,\a^\lambda)$.  Indexed by positive rational numbers $\lambda$,
this family forms a nested sequence of ideals.  The
 values of $\lambda$ where the multiplier ideals change are known as jumping numbers.  These
discrete local invariants were studied systematically in \cite{Ein}, after appearing indirectly in \cite{Lib}, \cite{Loes}, \cite{Loes2}, and \cite{Loes3}.  The jumping numbers encode algebraic information about the ideal, and geometric properties of the associated closed subscheme.  Smaller jumping numbers can be thought to correspond to ``worse'' singularities.

\smallskip

In order to find the jumping numbers of an ideal, one must first undertake the difficult task of  resolving singularities. However, even when a resolution is readily available, calculating jumping numbers can be problematic.   Outside the work of Howald on monomial ideals in \cite{How1},
few examples have been successfully computed.
In the present article, we study jumping numbers on an algebraic surface with a rational singularity.  We will show how to find the jumping numbers of an ideal by understanding the numerics of a log resolution.
This result significantly improves our ability to compute these invariants,  and provides important new examples for the continuing study of jumping numbers.

\smallskip

Perhaps the most important application of our method lies in finding the jumping numbers of an embedded curve on a smooth surface.
While recent progress has been made along these lines in \cite{Tarmo}, the algorithm we present is simple to use and unique in that it applies to reducible curves.  Once the singularities of the curve have been resolved, the invariants are found by checking finitely many inequalities between intersection numbers on the resolution.  For emphasis, we will use plane curves in motivating examples throughout this article, and Section 6 is entirely devoted to applications on smooth surfaces.  Proposition \ref{app} shows, for example, that a complete finite-colength ideal in the local ring at a closed point of a smooth surface is simple if and only if 1 is not a jumping number of the ideal.
\smallskip

Our techniques build upon the work of Smith and Thompson in \cite{How}.
Roughly speaking, multiplier ideals are defined by a finite 
number of divisorial conditions on a given resolution. These conditions are by no means independent, and checking 
each of them is often unnecessary. We work to identify which of these conditions are essential near a change in the multiplier ideals. By 
reducing the number of conditions to check, we improve our ability to compute jumping numbers and
understand the information these invariants encode. 

\smallskip

To give a more detailed overview, we briefly summarize our definitions.   Let $\a$ be an ideal sheaf on a complex algebraic surface $X$ with a rational singularity.  Fix a log resolution $\pi :Y \to X$ of the pair $(X,\a)$, with $\a \O_Y = \O_Y(-F)$ and relative canonical divisor $\K$.  If $G$ is a reduced subdivisor of $F$, we say 
 $\lambda \in \Q_{>0}$ is a \emph{candidate jumping number} for $G = E_1 + \cdots + E_k$ when $\ord_{E_i}(\K - \lambda F)$ is an integer for all $i=1, \ldots, k$.  Every jumping number is necessarily a candidate jumping number for some $G$, while not every candidate jumping number is realized as a jumping number.   When a candidate jumping number $\lambda$ for $G$ is a jumping number, we say $\lambda$ is \emph{contributed} by $G$ if
\[
\J(X,\a^\lambda) =\pi_* \O_Y(\lceil \K - \lambda F \rceil) \neq \pi_* \O_Y(\lceil \K- \lambda F \rceil + G).
\]
This contribution is said to be \emph{critical} if, in addition, no proper subdivisor of $G$ contributes $\lambda$.  The content of Theorems \ref{lemma} and \ref{cor} is summarized below, showing how to identify the reduced exceptional divisors which critically contribute a jumping number.

\begin{introtheorem}
Suppose $\a$ is an ideal sheaf on a complex surface $X$ with an isolated rational singularity. Fix a log resolution $\pi: Y \to X$ with $\a\O_Y = \O_Y(-F)$, and a reduced divisor $G = E_1+ \cdots +E_k$  on $Y$ with exceptional support.
\begin{enumerate}[(i)]
\item 
The jumping numbers $\lambda$ critically contributed by $G$ are determined by the intersection numbers \mbox{$\lceil \K - \lambda F \rceil \cdot E_i$}, for $i = 1, \ldots, k$.
\item
If $G$ critically contributes a jumping number, then it is necessarily a connected chain of smooth rational curves.  The ends of $G$ must either
 intersect  three other prime divisors in the support of $F$, or
correspond to a Rees valuation of $\a$.
\end{enumerate}
\end{introtheorem}

Smith and Thompson \cite{How} originally defined jumping number contribution for prime divisors.
However, in order to account for all jumping numbers, it is not sufficient to consider prime divisors alone.
By extending the notion of contribution as above, we ensure that every jumping number is critically contributed by some reduced divisor.  This observation and the Theorem above are the basis for the algorithm presented in Section \ref{alg}.
Further, critical contribution by a reduced exceptional divisor with multiple components is readily observed.
If $G$ is a connected chain of exceptional divisors on a fixed birational modification $\pi$ of the plane over the origin, there is an ideal $\a$ having $\pi$ as a log resolution with a jumping number critically contributed by $G$.

\smallskip

The structure of this article is as follows.
In Section 2, we briefly recall the definitions of multiplier ideals, jumping numbers, and a rational singularities on algebraic surfaces.  We also review without proof the properties of each which will be necessary for our calculations.  Section 3 motivates and defines the notion of jumping number contribution by a reduced divisor on a fixed log resolution.  The numerical criteria for critical contribution
is derived in Section 4, the technical heart of the paper, and we deduce that a critically contributing exceptional collection is a chain of rational curves in Section 5.  At the beginning of Section 6, we outline the algorithm for computing jumping numbers.  The remainder of the section is devoted to two substantial applications, 
where we compute the jumping numbers of the maximal ideal at the singular point in a Du Val (Example~\ref{Du Val}) or toric surface singularity (Example~\ref{Toric}).  In the final section, we specialize to the smooth case in order to discuss various questions and applications.

\smallskip

We would like to thank Mattias Jonsson, Karen Smith, 
Karl Schwede, Alan Stapledon, and several others in the community at the University of Michigan for useful conversations and support.  

\section{Multiplier Ideals on Rational Surface Singularities}

We begin by fixing notation and recalling the precise definition of a multiplier ideal.  Unless explicitly altered, these conventions shall remain in effect throughout.
Let $R$ be the local ring at a closed point on a normal complex algebraic surface with function field $\mathcal{K} = \Frac(R)$.  Recall that $R$ is said to be a rational singularity if there exists a resolution of singularities $\pi : Y \to X = \spec(R)$ such that $ H^1(Y,\O_Y) = 0$.  The theory of rational singularities of algebraic surfaces was first developed by Artin in \cite{
Artin1} and \cite{Artin}, and studied extensively by Lipman in \cite{LipRat}.  We shall need various facts proved therein, and cite them without proof as necessary.

\smallskip

Recall that a log resolution of an ideal $\a \subseteq R$ is a proper birational morphism $\pi : Y \to X$ such that
\begin{itemize}
\item
Y is smooth, and  $\a \O_Y$ is the locally principal ideal sheaf of an effective divisor $F$;
\item
The prime exceptional divisors and the irreducible components of $F$ are all smooth and intersect transversely.
\end{itemize}
When $X$ is smooth, it is well known that these resolutions exist; further, any such is a composition of maps obtained by blowing up closed points.  The same is also true when $X$ has rational singularities; see Theorem 4.1 of \cite{LipRat}.

\smallskip

Since $X$ is normal, there is a well defined linear equivalence class of canonical Weil divisors.  We may choose a representative $K_X$ for this class by setting $K_X= \pi_* K_Y$, where $K_Y$ is a canonical divisor on $Y$.  It is shown in \cite{LipRat}, Proposition 17.1, that the divisor class group of $R$ is finite.  Hence, there is an integer $m > 0$ such that $mK_X$ is a Cartier divisor.  In particular, $\pi^*K_X = \frac{1}{m} \pi^*(mK_X)$ is a well-defined \mbox{$\Q$-divisor} on $Y$.  By construction, there is an exceptionally supported \mbox{$\Q$-divisor} $\K$  such that
\[
K_Y = \pi^* K_X + \K.
\]
One checks $K_\pi$ is independent of the choice of canonical divisor on $Y$, which we refer to as the relative canonical divisor.  In general, $K_\pi$ is neither integral nor effective; however, when $X$ is smooth, $\K$ is both as it is defined by the Jacobian determinant of $\pi$.

\smallskip

If $D = \sum d_j D_j$ is a \mbox{$\Q$-divisor}, set $\lfloor D \rfloor = \sum \lfloor d_j \rfloor D_j$ and $\{D\} = \sum \{d_j\} D_j$ to be the integer and fractional parts of $D$, respectively.  Note that $D = \lfloor D \rfloor + \{D\}$.  Further, let $\lceil D \rceil = - \lfloor - D \rfloor$ be the \mbox{round-up} of $D$.  In general, these operations do not behave well with respect to restriction or pull-back.

\begin{definition}  
The \emph{multiplier ideal} of the pair $(X,\a)$ with coefficient $\lambda \in \Q_{>0}$ is the ideal $
\J(X, \a^\lambda) = \pi_* \O_Y(\lceil \K - \lambda F \rceil)$ in $R$.  
\end{definition}

For an introduction in the smooth case, we refer the reader to~\cite{Blickle}.  One immediately checks that this definition is independent of the choice of log resolution; however, we will generally work with a fixed resolution.  Since $X$ has rational singularities, this is motivated in part by the existence of a minimal log resolution, \mbox{i.e.} a log resolution through which all others must factor.  See \cite{LipRat} for further discussion.

\smallskip

Multiplier ideals have emerged as a fundamental tool in algebraic and analytic geometry.  A detailed account of their properties, applications, and further references, may be found in~\cite{Laz}.  
Here we briefly mention two important results to which we will refer in later sections. 

\begin{itemize}
\item  The \emph{local vanishing} property of multiplier ideals states that  $R^j\pi_* \O_Y( \lceil \K - \lambda F \rceil) = 0$ for all $j > 0$, and all $\lambda \in \Q_{>0}$.
\item According to \emph{Skoda's Theorem}, $\J(X, \a^\lambda) = \a \, \J(X, \a^{\lambda - 1})$ for all $\lambda > \dim(R) = 2$.
\end{itemize}

\noindent 
The first result essentially follows from the vanishing theorems of Kawamata and Viehweg,\footnote{
See generalization 9.1.22 of  \cite{Laz} for a sketch of the proof of local vanishing in this generality, or \mbox{Theorem~1-2-3} of \cite{KMM87} for a complete argument.  The original vanishing theorems appear in  \cite{Kawamata82} and \cite{Viehweg82}.}
 and can be used give a proof of the second statement.

\smallskip

By identifying the function fields of $X$ and $Y$, each  prime divisor $E$ appearing in $\K$ or $F$ corresponds to a discrete valuation $\ord_E$ on $\mathcal K = \Frac(R)$ with value group $\Z$. To check whether a function $f \in R$ is in $\J(X, \a^\lambda)$, one must show for all such $E$ that
\begin{equation}
\label{ord}
\ord_E f \geq \ord_E(\lfloor \lambda F - \K \rfloor).
\end{equation}
Consider what happens as one varies $\lambda$.  Increasing $\lambda$ slightly does not change (\ref{ord}); thus, $\J(X, \a^\lambda) = \J(X, \a^{\lambda + \epsilon})$ for sufficiently small $\epsilon > 0$.  However, continuing to increase $\lambda$ further will cause the coefficient of $E$ in $\lfloor \lambda F -\K \rfloor$ to change, and this sometimes results in a jump in the mutliplier ideals $\J(X, \a^\lambda)$.

\begin{definition}
We say that $\lambda \in \Q_{>0}$ is a \emph{candidate jumping number} for a prime divisor $E$  appearing in $F$ if the order of vanishing necessary for membership \eqref{ord} in the multiplier ideals of $(X,\a)$ changes at $\lambda$, \mbox{i.e.} $\ord_E(\lambda F - \K)$ is an integer.  If $G$ is a reduced divisor on $Y$,  a candidate jumping number for $G$ is a common candidate jumping number for the prime divisors in its support.
The coefficient $\lambda \in \Q_{>0}$ is a  \emph{jumping number} of the pair $(X, \a)$ if $\J(X,\a^{\lambda- \epsilon}) \neq \J(X,\a^\lambda)$ for all $\epsilon > 0$.  The smallest jumping number is the \emph{log canonical threshold} of the pair $(X,\a)$.
\end{definition}

Since $X$ is normal, note that condition \eqref{ord} is trivial for $\ord_E(\lfloor \lambda F - \K \rfloor) \leq 0$.  
We see explicitly that the nontrivial candidate jumping numbers for $E$ are $\{\frac{\ord_E \K + m}{\ord_E F} : m \in \Z_{>0}\}$.  
The jumping numbers of $(X,\a)$ are in general strictly contained in the union of the candidate jumping numbers of all of the prime divisors appearing in $F$.  In particular, they form a discrete set of invariants. 
Furthermore, by Skoda's Theorem, the jumping numbers of a pair $(X,\a)$ are eventually periodic; $\lambda > 2$ is a jumping number if and only if $\lambda - 1$ is a jumping number.

\section{Jumping Numbers Contributed by Divisors}
\label{Defs}

In order to compute the jumping numbers of $(X, \a)$ from a log resolution $\pi:Y \to X$, we must
first understand
the causes in the underlying jumps of the multiplier ideals.  To this end, the following definitions allow us to attribute the appearance of a jumping number to certain reduced divisors on $Y$.

\begin{definition} 
\label{chain}
Let $G$ be a reduced divisor on $Y$ whose support is contained in the support of $F$.  We will say $G$ \emph{contributes} a candidate jumping number $\lambda$ if 
\[
\J(X, \a^\lambda) \subsetneq \pi_* \O_{Y}(\lceil K_\pi - \lambda F \rceil  + G).
\]
This contribution is said to be \emph{critical} if, in addition, no proper subdivisor of $G$ contributes $\lambda$, i.e.
\[
\J(X, \a^\lambda) = \pi_* \O_{Y}(\lceil K_\pi - \lambda F \rceil  + G')
\]
for all divisors $G'$ on $Y$ such that $0 \leq G' < G$.
\end{definition}

Note that this is an extension of Definition 5 from \cite{How}, where Smith and Thompson introduced
jumping number contribution for prime divisors.
Further, if a jumping number is contributed by a prime divisor $E$, this contribution is automatically critical.  It is easy to see every jumping number is critically contributed by some reduced divisor on $Y$.   The following example illustrates the original motivation for defining jumping number contribution.

\begin{example}
\label{firstexample} 
Suppose $R$ is the local ring at the origin in $\A^2$, and $C$ is the germ of the analytically irreducible curve defined by the polynomial $x^{13} - y^5 = 0$.  The minimal log resolution $\pi : Y \to X$ of $C$ is a sequence of six blow-ups along closed points (there is a unique singular point on the transform of $C$ for the first three blow-ups, after which it takes an additional three blow-ups to ensure normal crossings).  If $E_1, \ldots, E_6$ are the exceptional divisors created, one checks
\[
\pi^*C =C+  5E_1 + 10 E_2 + 13E_3 + 25E_4 + 39 E_5 + 65 E_6
 \qquad
 \K = E_1 +2E_2 + 3E_3 + 6E_4 + 10E_5 + 17E_6.
 \]
 Thus, the nontrivial candidate jumping numbers of $E_1$ are $\{ \frac{1 + m}{5} : m \in \Z_{>0}\}$, whereas those for $E_6$ are $\{ \frac{17 + m}{65} : m \in \Z_{>0}\}$.  One can 
 compute%
 \begin{footnote}{The polynomial $f(x,y) = x^{13} - y^5$ is nondegenerate with respect to its Newton polyhedron, and thus it is a theorem of Howald \cite{How2} that the jumping numbers of $f$ less than 1 coincide with those of its term ideal $(x^{13},y^5)$.  One may then use the explicit formula \cite{How1} for the jumping numbers of a monomial ideal to achieve the desired result.  This argument is essentially repeated in Example 3.6 of \cite{Ein}, and discussed at greater length in Section 9.3.C of \cite{Laz}.  Note that since this curve is analytically irreducible, the result also follows from \cite{Tarmo}.
It is also possible to use the numerical results of Section \ref{main} to check this directly.}
 \end{footnote} that  the jumping numbers of the pair $(\A^2,C)$ are precisely
 \[
 \left\{ \frac{13(r+1) + 5(s+1)}{65} + t \quad  \Big| \quad r,s,t \in \Z_{\geq0} \mbox{ and } \frac{13(r+1) + 5(s+1)}{65}<1\right\} \, \,\cup \, \, \Z_{>0}.
 \]
Note that the jumping numbers less than one are all candidate jumping numbers for $E_6$, but for no other $E_i$.  Thus, for any jumping number $\lambda < 1$  and sufficiently small $\epsilon>0$, we have
\[
\J(X, \lambda C) \subsetneq \pi_*\O_X(\lceil \K -  \lambda \pi^*C \rceil + E_6)  = \J(X, (\lambda -\epsilon)C).
\]
In other words, the jump in the multiplier ideal at $\lambda$ is due solely to the change in condition (\ref{ord}) along $E_6$.  According to Definition \ref{chain}, all of the jumping numbers less than one are contributed by $E_6$, and are not contributed by any other divisor.  
\end{example}

In general, however, the situation is often far less transparent. Distinct prime divisors often have common candidate jumping numbers.  In some cases, as the next example from \cite{How} shows, these prime divisors may separately contribute the same jumping number.  In others, collections of these divisors may be needed to capture a jump in the multiplier ideals.

\begin{example}
\label{secondexample}
Suppose $R$ is the local ring at the origin in $\A^2$, and $C$ is the germ of the plane curve defined by the polynomial $(x^3 - y^2)(x^2 - y^3) = 0$ at the origin.  The minimal log resolution $\pi$ has five exceptional divisors: $E_0$ obtained from blowing up the origin; $E_1$ and $E_1'$ obtained by blowing up the two intersections of $E_0$ with the transform of the curve $C$ (both points of tangency); and $E_2$ (respectively $E_2'$) obtained by blowing up the intersection of the three smooth curves $C$, $E_0$, and $E_1$ (respectively, the three smooth curves $C$, $E_0$, and $E_1'$).  One checks
\[
\pi^*C = C + 4E_0 + 5(E_1 + E_1') + 10(E_2 + E_2') 
\qquad
\K =  E_0 + 2(E_1 + E_1') + 4(E_2 + E_2')
\]

\smallskip

\begin{center}
 \includegraphics[width=1.3in]{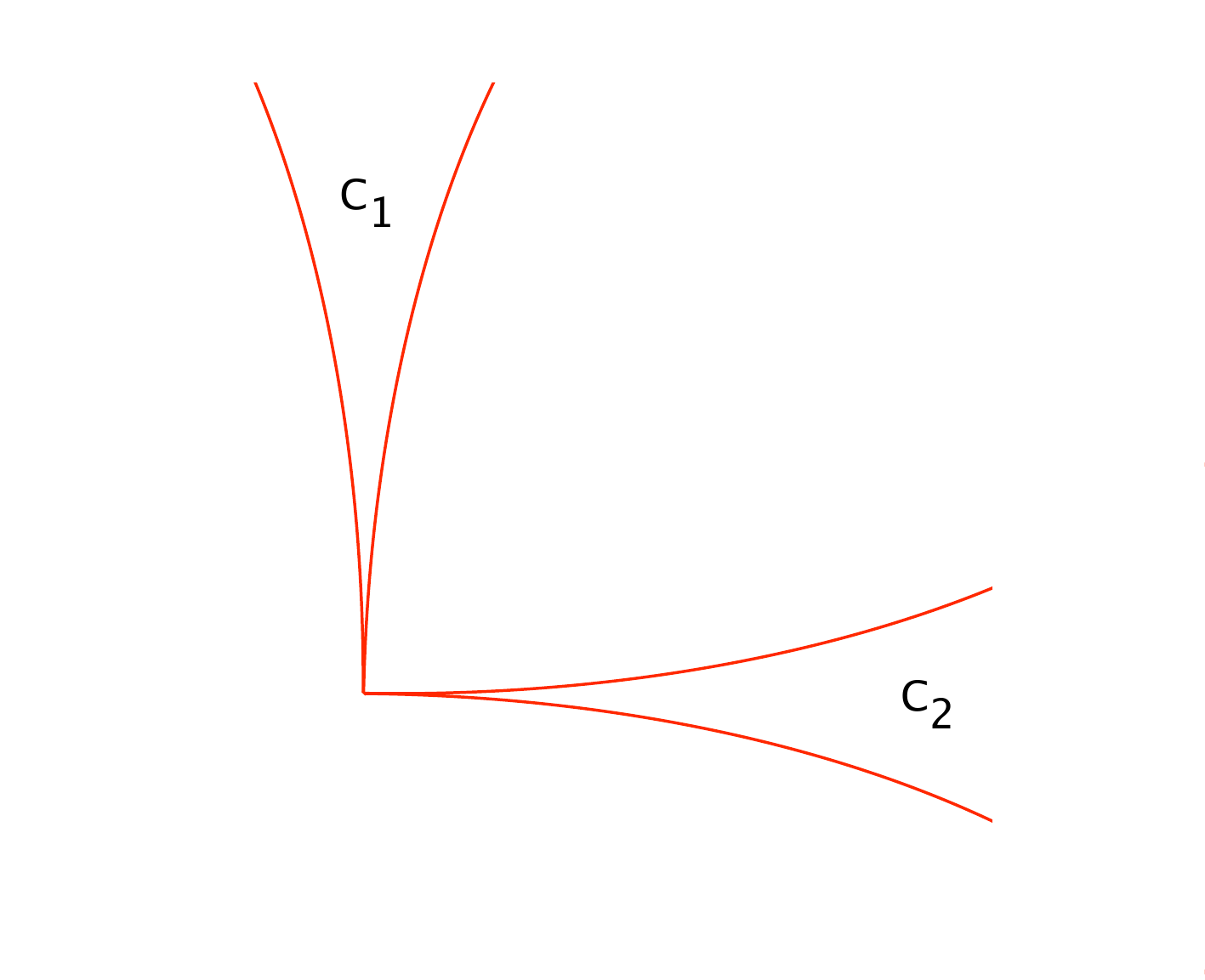}
 \raisebox{.55in}{$\arrow{w}$} \hspace{-.1in}
  \includegraphics[width=1.3in]{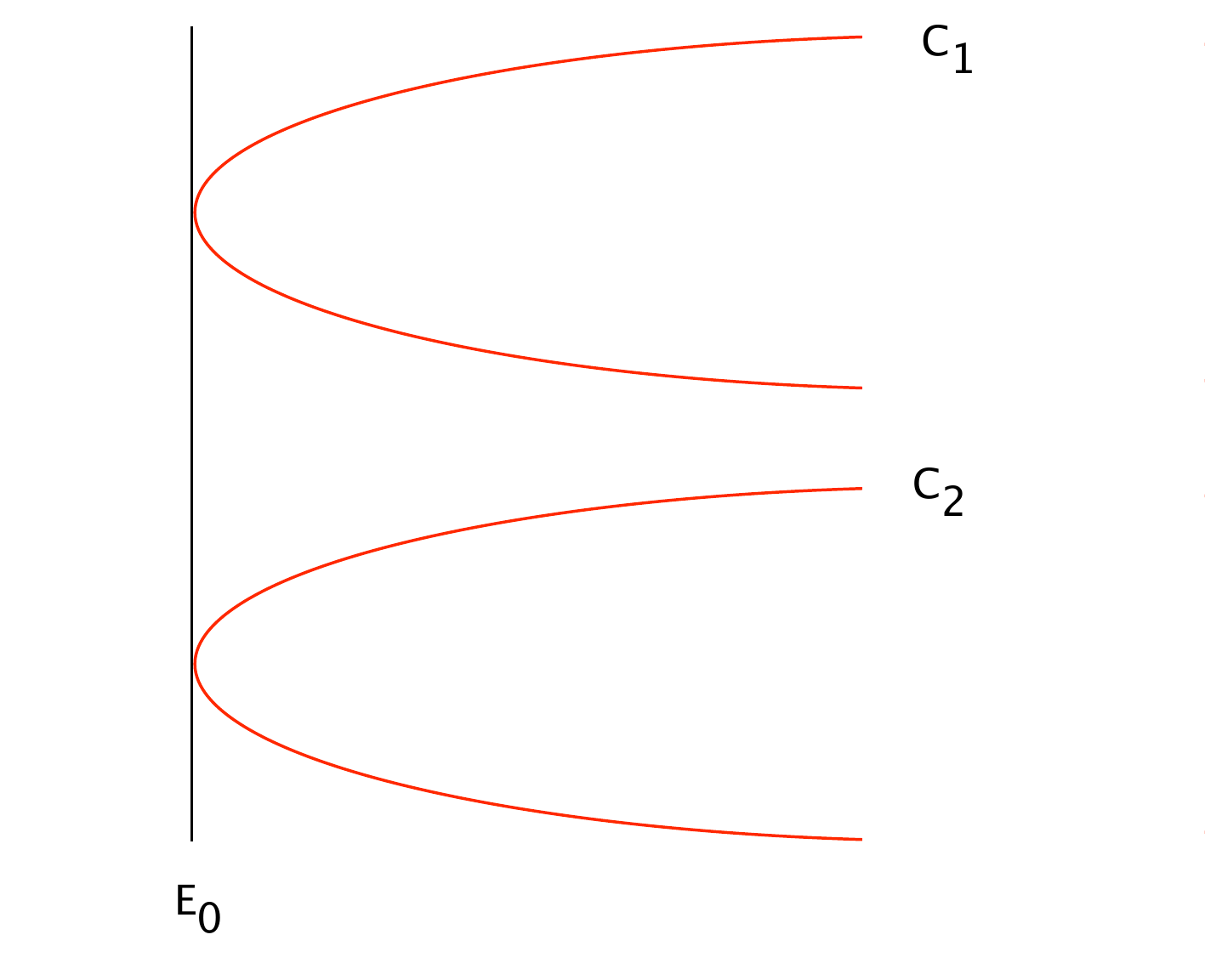}
 \hspace{.1in}
 \raisebox{.55in}{$\arrow{w}$} \hspace{-.1in}
  \includegraphics[width=1.3in]{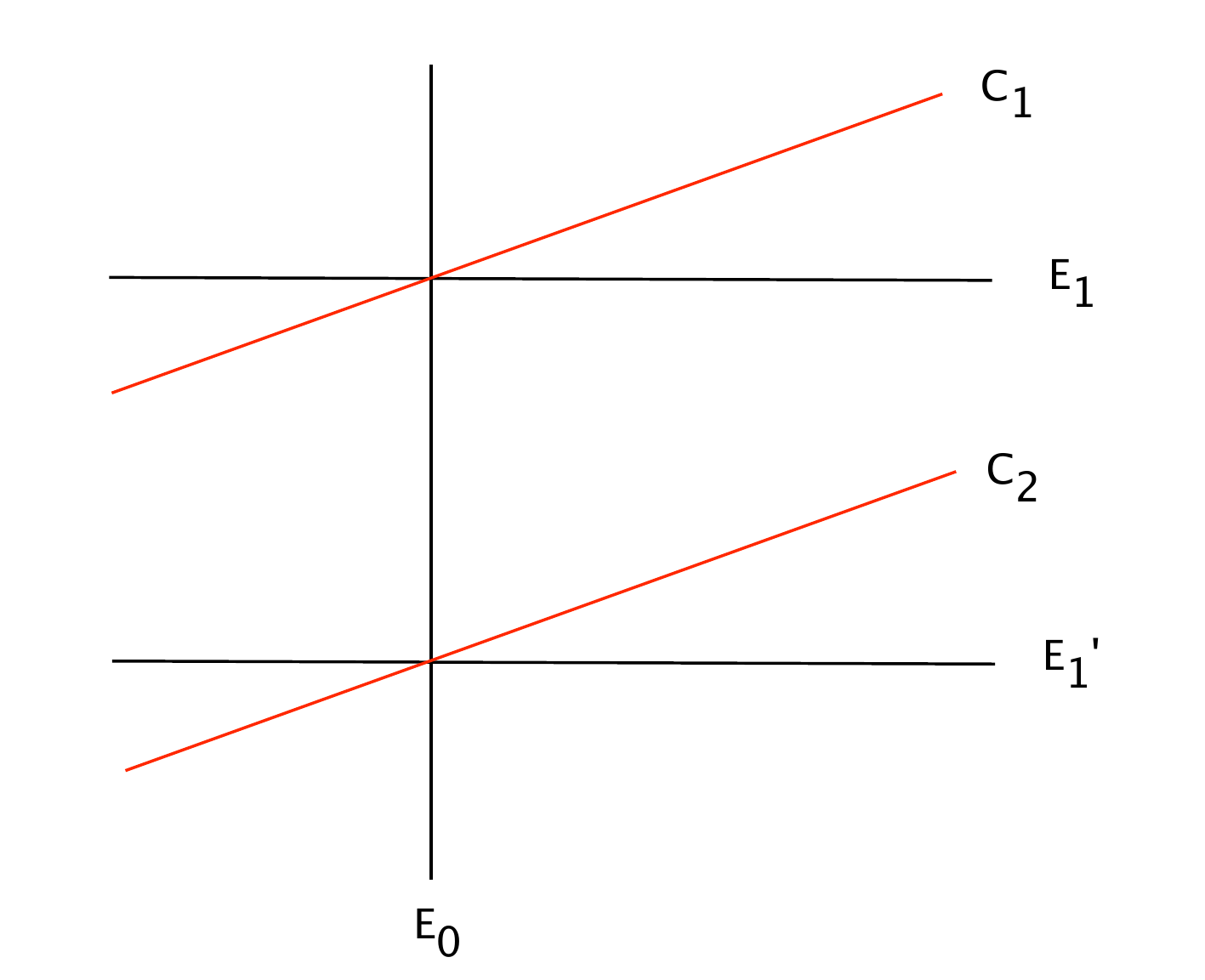}
 \hspace{.1in}
 \raisebox{.55in}{$\arrow{w}$} \hspace{-.1in}
  \includegraphics[width=1.3in]{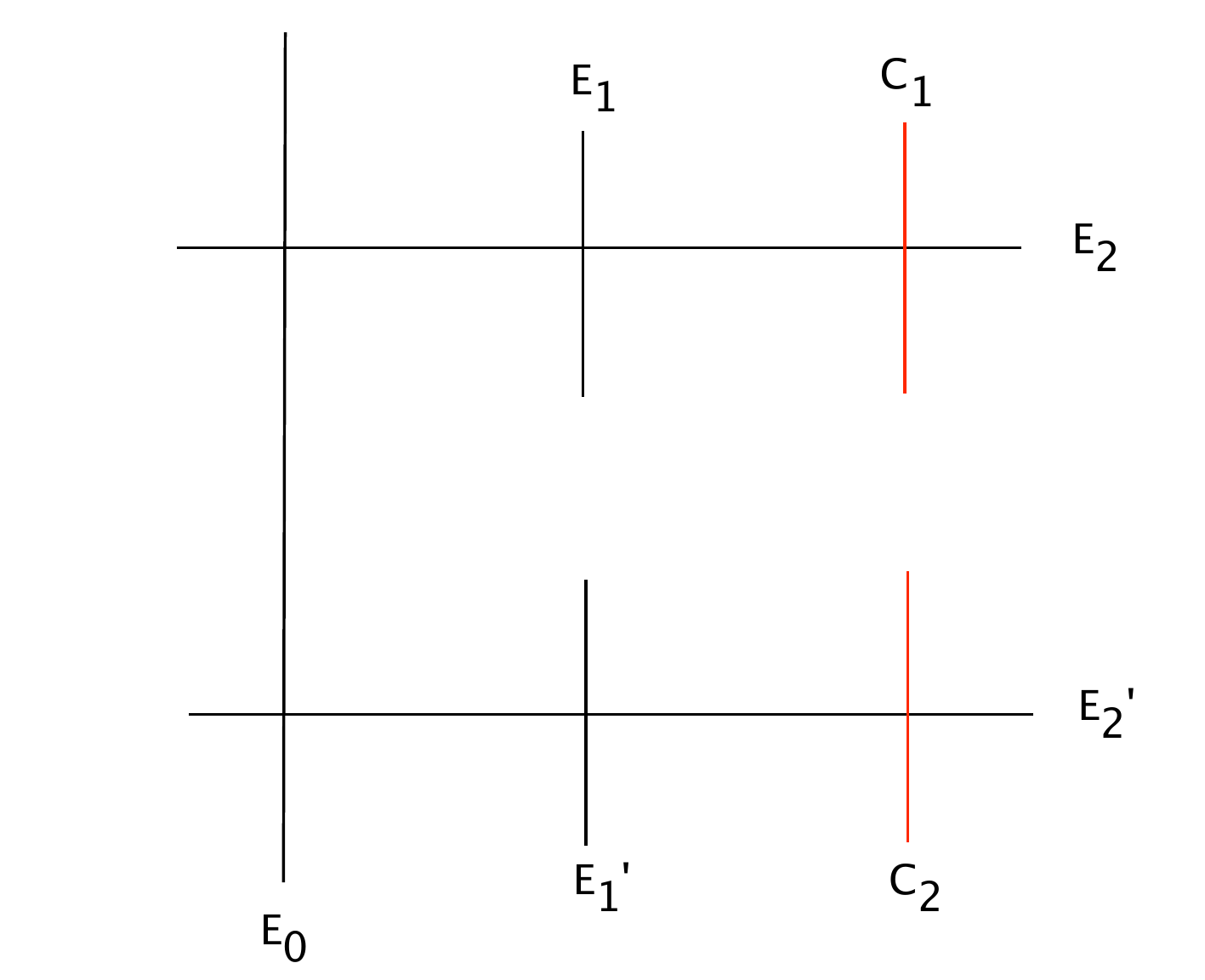}
  \end{center}
  
 \smallskip

\noindent so that the log canonical threshold is $\frac{1}{2}$.  Here, we have $\lceil \K - \frac{1}{2} \pi^*C \rceil =  -E_0 - E_2 - E_2'$, so that the three new conditions for membership in $\J(\frac{1}{2} C)$ are vanishing along $E_0,E_2, E_2'$.  However, and herein lies the problem in determining the precise cause of the jump in the multiplier ideal, these are not independent conditions.  Requiring vanishing along any of these three divisors automatically guarantees vanishing along the others. Thus, instead of attributing the jump to any prime divisor, it seems natural to suggest that the collection $E_0 + E_2 + E_2'$ is responsible.  According to Definition \ref{chain}, $E_0 + E_2 + E_2'$ critically contributes~$\frac{1}{2}$.  
Further, it is shown in~\cite{How} that $\frac{9}{10}$ is a jumping number contributed by either $E_2$ or $E_2'$.  One may even argue there is a sense in which the collection $E_2 + E_2'$ is responsible for this jump.  Indeed, for sufficiently small $\epsilon > 0$, we have
\[
\J(X, (\frac{9}{10} -  \epsilon)C) \subsetneq \pi_*\O_X(\lceil \K -  \frac{9}{10} \pi^*C \rceil + E_2) \subsetneq \pi_*\O_X(\lceil \K -  \frac{9}{10} \pi^*C \rceil + E_2+E_2')  = \J(X, \frac{9}{10}C).
\]
In this case, the jumping number $\frac{9}{10}$ is contributed by $E_2 + E_2'$; however, this contribution is not critical as either $E_2$ or $E_2'$ also contribute $\frac{9}{10}$.

\end{example}

\begin{remark}
Contribution and critical contribution are somewhat subtle to formulate valuatively.  If $G = E_1 + \cdots + E_k$ critically contributes $\lambda$ to $(X,\a)$, one can show there is some $f \in R$ which is not in $\J(X,\a^\lambda)$ because it fails to satisfy condition \eqref{ord} precisely along  $E_1, \ldots, E_k$, and $G$ is a minimal collection with this property.  This depends not only on  the divisorial valuations appearing in $G$, but all those appearing in $F$.  In particular, there is no reason to believe this is independent of the chosen resolution.  However, when $X$ is smooth, it is possible to formulate a notion of contribution which is model independent by considering all possible resolutions simultaneously.  Explicitly,
it is shown in \cite{Mattias} that the dual graphs of all  resolutions fit together in a nice way to give the so-called valuative tree, and a reduced effective divisor on $Y$ corresponds  in a natural way to a union of subtrees of the valuative tree. Similar ideas were explored in \cite{Mattias2}.
\end{remark}

\section{Numerical Criterion for Critical Contribution}
\label{NumCrit}

We now begin working towards a numerical test for jumping number contribution.  The first step is to interpret contribution cohomologically.

\begin{proposition} 
Suppose that $\lambda$ is a candidate jumping number for the reduced divisor $G$.  Then $\lambda$ is realized as a jumping number for $(X,\a)$ contributed by $G$ if and only if
\[
H^0 (G,(\lceil K_\pi - \lambda F \rceil  + G)\vert_G) \neq 0.
\]
Furthermore, this contribution is critical if and only if we have 
\[
H^0(G', \left. (\lceil K_\pi - \lambda F \rceil  + G')\right|_{G'}) = 0
\]
for all divisors $G'$ on $Y$ such that $0 \leq G' < G$.
\end{proposition}

\begin{proof}
Consider the short exact sequence
\[
0 \arrow{e} \O_Y(\lceil K_\pi - \lambda F \rceil) \arrow{e} \O_Y(\lceil K_\pi - \lambda F \rceil  + G)  \arrow{e} \O_G(\left. (\lceil K_\pi - \lambda F \rceil  + G)\right|_G) \arrow{e} 0.
\]
on $Y$. Pushing down to $X$, we arrive at
\[
0 \arrow{e} \J(X, \a^\lambda) \arrow{e} \pi_* \O_{Y}(\lceil K_\pi - \lambda F \rceil  + G)
\arrow{e} \pi_*\O_{G}(\left. (\lceil K_\pi - \lambda F \rceil  + G)\right|_G) \arrow{e} R^1\pi_* \O_Y(\lceil K_\pi - \lambda F \rceil) \arrow{e} \cdots
\]
However, local vanishing for multiplier ideals guarantees $R^1\pi_* \O_Y(\lceil K_\pi - \lambda F \rceil) = 0$.  In particular, we see that $G$ contributes a common candidate jumping number $\lambda$ for $E_{1}, \ldots, E_{k}$  to the pair $(X,\a)$ if and only if $\pi_* \O_{G}((\lceil K_\pi - \lambda F \rceil  + G)\vert_G) = H^0 (G,(\lceil K_\pi - \lambda F \rceil  + G)\vert_G) \neq 0$.  The second statement is immediate from the definition of critical contribution.
\end{proof}

\begin{corollary}
If $G$ critically contributes a jumping number $\lambda$, then $G$ is connected.
\end{corollary}

\begin{proof}
By way of contradiction, suppose
we may write $G = G' + G''$ giving a separation, where $0 < G', G'' < G$ and $G', G''$ are disjoint.  Then we have
\[
H^0(G,(\lceil K_\pi - \lambda F \rceil  + G)\vert_G) = H^0(G',(\lceil K_\pi - \lambda F \rceil  + G')\vert_{G'}) \oplus H^0(G'',(\lceil K_\pi - \lambda F \rceil  + G'')\vert_{G''}).
\]
Thus, if $G$ contributes a jumping number $\lambda$ to the pair $(X,\a)$, either $G$ or $G'$ must also contribute $\lambda$.  In particular, $G$ does not critically contribute $\lambda$.
\end{proof}

Suppose now that $G$ is a reduced divisor on $Y$ with exceptional support.  The prime exceptional divisors of $\pi$ are all smooth rational curves intersecting transversely, and there are no loops\footnote{When $X$ is smooth, this statement can be shown by induction on the number of blow-ups in $\pi$.  More generally, Proposition 1  of \cite{Artin1}  states that rational singularities are equivalent to $p_a(Z) \leq 0$ for all effective exceptional divisors $Z$, where $p_a(Z) = 1 - \chi(Z)$ denotes the arithmetic genus.} of exceptional divisors.
We therefore assume that  $G= E_1 + \cdots + E_k$ is a nodal tree of smooth rational curves.  
A global section $s$ of $\O_G(\left.(\lceil K_\pi - \lambda F \rceil  + G)\right|_G)$ is equivalent to  a collection of global sections $s_{j}$ of $\O_{E_{j}}(\left.(\lceil K_\pi - \lambda F \rceil  + G)\right|_{E_{j}})$ for $j = 1, \ldots, k$ which agree on the intersections.   Indeed, this statement is easy verified for two rational curves intersecting transversely, and the general case follows by induction on $k$.  Since the existence of nonzero global sections on smooth rational curves is equivalent to having non-negative degree, we now show critical contribution by reduced exceptional divisors can be checked numerically.  When $G$ is prime and $X$ is smooth, this criterion was given in \cite{How}.

\begin{theorem}
\label{lemma}
Denote by $R$ the local ring at an isolated rational singularity on a normal complex surface.  Let $\a \subseteq R$ be an ideal, and $\pi: Y \to X= \spec(R)$ a log resolution of $(X,\a)$ such that $\a\O_Y = \O_Y( - F)$.  Suppose that $\lambda$ is a candidate jumping number for the reduced divisor $G$ with connected exceptional support.
\begin{itemize}
\item
If $G=E$ is prime, then $\lambda$ is (critically) contributed by $E$ to $X$
if and only if
\[
\lceil \K - \lambda F \rceil \cdot E\geq - E\cdot E.\]
\item
If $G$ is reducible, then
$\lambda$ is critically contributed by $G$ if and only if 
\[ \lceil \K - \lambda F \rceil \cdot E = -G\cdot E \]
 for all prime divisors $E$ in the support of $G$.
\end{itemize}
\end{theorem}

\begin{proof}
Suppose first $G = E$ is a single prime exceptional divisor.  Then $\lambda$ is contributed by $E$ if and only if $H^0( E, \left. (\lceil \K - \lambda F \rceil + E) \right|_E) \neq 0$.  Since $E \cong \P^1$, it is equivalent that this line bundle have non-negative degree, \mbox{i.e.} $\lceil \K - \lambda F \rceil \cdot E\geq - E\cdot E$.

\smallskip

Thus, we assume $G = E_1 + \cdots  + E_k$ is reducible.
Note that the numerical conditions given are clearly sufficient.  They are equivalent to saying 
$\O_G(\left.(\lceil K_\pi - \lambda F \rceil  + G)\right|_G)$ restricts to the trivial bundle 
on each of $E_1, \ldots, E_k$,  and hence must be the trivial bundle\begin{footnote}{This is true whenever $G$ is effective with exceptional support.
Theorem 1.7 of \cite{Artin} concludes that the isomorphism class of a line bundle on $G$ is determined by the degrees of its restrictions to $E_1, \ldots, E_k$.}\end{footnote} on $G$.  In particular, 
$H^0(G,\left.(\lceil K_\pi - \lambda F \rceil  + G)\right|_{G}) \neq 0$, and
$G$ contributes $\lambda$ to $(X,\a)$.  To see this contribution is critical, note that if $0 \leq G' < G$, then
the degree of $\left. \O_{G'}(\left.(\lceil K_\pi - \lambda F \rceil  + {G'})\right|_{G'})\right|_{E_i}$ along $E_i$ is
$- E_i \cdot (G-G')$.  In particular, the sections of $\O_{G'}(\left.(\lceil K_\pi - \lambda F \rceil  + {G'})\right|_{G'})$ are identically zero when restricted to  to any component $E_i$ of $G'$ which intersects $G - G'$, and are constant along any other component of $G'$.
Since $G$ was connected, 
one sees any global section must be identically zero. 

\smallskip

Now, assume $G$ critically contributes $\lambda$ to $(X,\a)$, and let $s \in H^0(G,\left.(\lceil K_\pi - \lambda F \rceil  + G)\right|_{G})$ be nonzero.  There is some $E$ in $\{ E_1, \ldots, E_k \}$ such that
$s|_E$ is nonzero.  In particular, we see that the restriction of $\O_G(\left.(\lceil K_\pi - \lambda F \rceil  + G)\right|_G)$ to $E$ has non-negative degree. 
 Suppose, by way of contradiction,
its degree is strictly positive.
Partition $G - E$ into its connected components, i.e. write $G - E = B_1+ \cdots + B_r$ where each $B_i$ for $1 \leq i \leq r$ is  the sum of all of the prime divisors in some connected component of $G - E$.  
Since $G$ is a nodal tree, we have that 
$0 < B_i \leq G - E$ and $B_i \cdot E = 1$  for each $i = 1, \ldots, r$. Furthermore, observe  that the supports of $B_1, \ldots, B_r$ are pairwise disjoint.
 Let $p_1, \ldots, p_r$ be the intersection points of $B_1, \ldots, B_r$ with $E$, respectively.
Re-indexing if necessary, choose a point $q \in E \setminus \{p_2, \ldots, 
  p_r \}$ such that $s(q) = 0$.
  \smallskip
  
We will show that $G' = G - B_1$ contributes $\lambda$ by proving $H^0(G',\O_{G'}(\left.(\lceil K_\pi - \lambda F \rceil  + G')\right|_{G'}) \neq 0$.  For $i \neq 1$,  we have $\left. (\lceil K_\pi - \lambda F \rceil  + G')\right|_{B_i} = \left.(\lceil K_\pi - \lambda F \rceil  + G)\right|_{B_i}$ since the supports of $B_1$ and $B_i$ are disjoint.  In particular, we may consider $s|_{B_i}$ as a global section of $\O_{B_i}(\left.(\lceil K_\pi -  \lambda F \rceil  + G')\right|_{B_i})$.  Next, identify $s|_E$ with a nonzero homogeneous polynomial on $\P^1$ of strictly positive degree.  Since $\deg(\O_E\left.(\lceil K_\pi - \lambda F \rceil  + G')\right|_{E}) =   \deg(\O_E\left.(\lceil K_\pi - \lambda F \rceil  + G)\right|_{E}) -1$,
removing one of its linear factors corresponding to a zero at $q$ yields a nonzero global section $t$ of $\O_{E}(\left.(\lceil K_\pi - \lambda F \rceil  + G')\right|_{E})$. By construction, $t(p_i) \neq 0$ if and only if $s(p_i) \neq 0$ for $2 \leq i \leq r$.  After scaling each $s|_{B_i}$ to agree with $t$ at $p_i$, we may glue to obtain a nonzero global section of $\O_{G'}(\left.(\lceil K_\pi - \lambda F \rceil  + G')\right|_{G'})$.  But this is absurd, as it implies that $G'$ contributes $\lambda$ to the pair $(X,\a)$.  Hence, we must have that $
 \deg(\O_E\left.(\lceil K_\pi - \lambda F \rceil  + G)\right|_{E}) = 0$.  Furthermore, nonzero global sections of $\O_E\left.(\lceil K_\pi - \lambda F \rceil  + G)\right|_{E}$ never vanish.  As $s$ does not restrict to zero along any component of $G$ which intersects $E$, the same arguments apply.  Using that $G$ is connected, the theorem now follows.
 \end{proof}
 
 \begin{example}
 Suppose $R$ is the local ring at the origin in $\A^2$, and $C$ is the germ of the plane curve defined by the polynomial  $(y-x^2)(y^2-x^5) = 0$.  The minimal log resolution $\pi$ is a sequence of four blow-ups along closed points (there is a unique singular point on the transform of $C$ for the first two blow-ups, after which it takes an additional two blowups to ensure normal crossings), and is pictured below.  If $E_1, \ldots, E_4$ are the exceptional divisors created, one checks
 \[
\pi^*C =C+  3E_1 + 6E_2 + 7E_3 + 14E_4
 \qquad
 \K = E_1 +2E_2 + 3E_3 + 6E_4
 \]
 
 \smallskip
 
 \begin{center}
 \raisebox{.2in}{
 \includegraphics[width=1.5in]{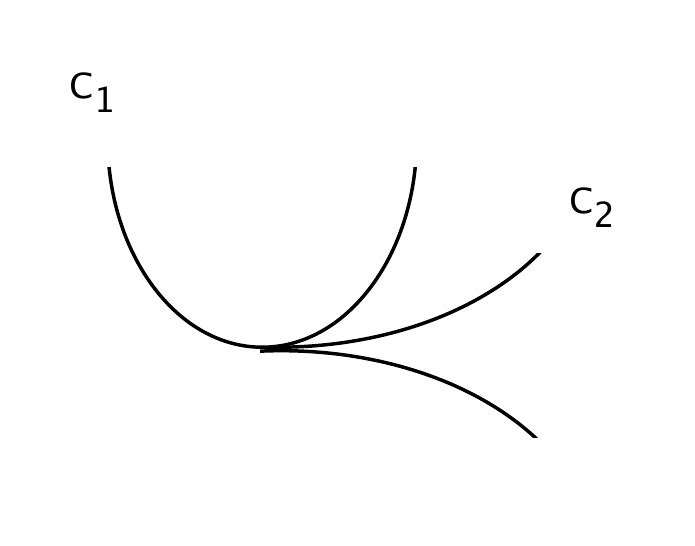}
 \hspace{.3in}
 \raisebox{.55in}{$\arrow{w}$} \hspace{-.1in}}
  \includegraphics[width=1.5in]{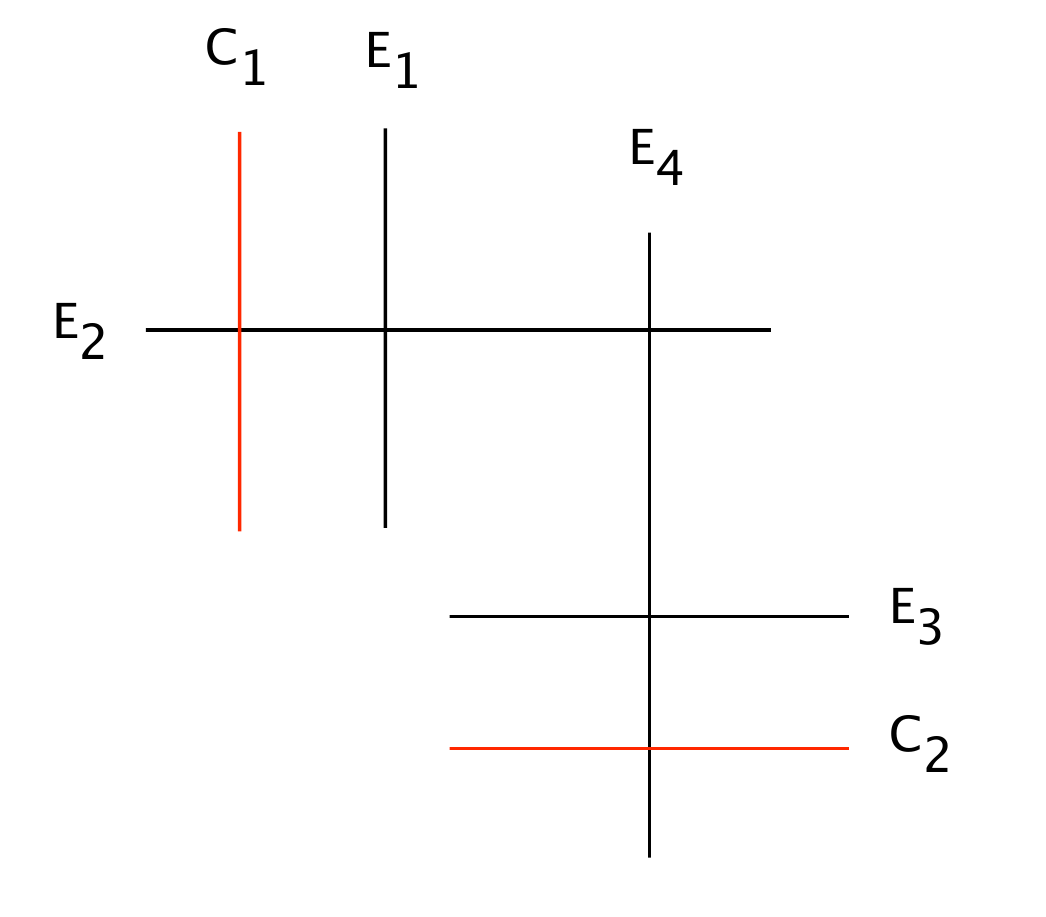}
 \end{center}
 
 \smallskip
 
 \noindent
 The only candidate jumping number less than one shared by both $E_2$ and $E_4$ is $\frac{1}{2}$.  One now computes directly that
 $ \lceil \K - \frac{1}{2} \pi^*C \rceil = -E_2 -E_4$,
and Theorem \ref{lemma} now implies that $E_2 + E_4$ critically contributes the jumping number $\frac{1}{2}$.  In Section \ref{alg}, we will discuss how the numerical criteria in Theorem \ref{lemma} give an algorithm for numerically computing all of the jumping numbers in such examples.  However, we postpone further discussion until after we have  examined which collections of exceptional divisors have the potential to critically contribute jumping numbers.
\end{example}

\section{Geometry of Contributing Collections}
\label{main}

We first recall some of the theory of complete ideals relevant to our calculations. If $R$ is the local ring at an isolated rational singularity of a normal complex surface, and $\a$ is an ideal of $R$, one can define its integral closure $\overline{\a}$ as follows.  Let $\nu : W \to X = \spec(R)$ be the normalized blow-up of $X$ along $\a$; then $\a \O_W = \O_W(-D)$ is a locally principal ideal sheaf on $W$ cutting out an effective divisor $D$.  One then sets $\overline{\a} = \nu_*\O_W(-D)$.  The valuations corresponding to prime divisors in the support of $D$ are said to be the \emph{Rees valuations} of the ideal $\a$.  When $\a = \overline{\a}$, we say that $\a$ is integrally closed or, more classically, \emph{complete}.

\smallskip

Since any log resolution $\pi: Y \to X$ of $\a$ factors through the normalized blow-up, we always have that  $\overline{\a} = \pi_*\O_Y(-F)$ where $\a \O_Y = \O_Y(-F)$.  In addition, we may characterize the exceptional Rees valuations of $\a$ numerically as follows.  Since $\a \O_Y = \O_Y(-F)$ is globally  generated, so is $\O_E( - F|_E)$ for any prime exceptional divisor $E$. In particular, we have $F \cdot E \leq 0 $.  Lemma  21.2 of \cite{LipRat} shows that $F \cdot E < 0$ if and only if $E$ corresponds to a Rees valuation of $\a$.

\smallskip

In \cite{How}, it was shown that a prime exceptional divisor on the minimal resolution of a curve on a smooth surface contributes a jumping number if and only if it intersects at least three other components of the support of the pull-back of the curve. The following theorem gives analogous restrictions to critically contributing collections in our setting.

\begin{theorem}
\label{cor}
Suppose $R$ is the local ring at an isolated rational singularity of a normal complex surface.  Let $\a \subseteq R$ be an ideal, and $\pi: Y \to X= \spec(R)$ a log resolution of $(X,\a)$ such that $\a\O_Y = \O_Y( - F)$. 
If  the reduced  divisor  $G$ with exceptional support
critically contributes the jumping number $\lambda$ to the pair $(X,\a)$, then $G$ is a connected chain.
The ends $E$ of $G$ must either:
  \begin{itemize}
 \item
 intersect at least three other prime divisors in the support of either $F$ or $\K$, or;
 \item
correspond to a Rees valuation of $\a$.
\end{itemize}
Furthermore, the non-ends of $G$ can intersect only those components of the support of $F$ 
that also have $\lambda$ as a candidate jumping number, and never correspond to a Rees valuation of $\a$.
\end{theorem}

\begin{proof}  
We will use the numerical criteria for critical contribution given in Theorem \ref{lemma}.
These are stated in terms of intersections with $\lceil \K - \lambda F \rceil$, which we manipulate into the following form
\[
\lceil \K - \lambda F \rceil = - \lfloor \lambda F - \K \rfloor = \K - \lambda F + \{ \lambda F - \K \}.
\]

Suppose first $G = E$ is a prime exceptional divisor, and $E$ is not a Rees valuation of $\a$.  Then by Theorem \ref{lemma}, since $E$ contributes $\lambda$, we have that $\lceil \K - \lambda F \rceil \cdot E \geq - E \cdot E$.  Plugging in from above and using that $F \cdot E = 0$, we have
\[
\{ \lambda F - \K \} \cdot E \geq  2,
\]
where we have made use of the adjunction formula 
\begin{equation}
\label{jim}
-\deg \left. (\K + E) \right|_E = -\deg K_E=2
\end{equation}
 applied to $E \cong \P^1$.  Since $\lambda$ is necessarily a candidate jumping number for $E$, it does not appear in $\{ \lambda F - \K \}$, which is an effective divisor with coefficients strictly less than one.  As the divisors in $\K$ and $F$ intersect transversely, at least three of them must intersect $E$ in order for the above inequality to hold.

Assume now $G$ is reducible.  Since $\lambda$ is critically contributed by $G$, we have that $G$ is connected and
$\lceil \K - \lambda F \rceil \cdot E = - G \cdot E$ for all $E$ in the support of $G$.  Plugging in and rearranging terms as above gives
\begin{equation}
\label{bob}
\{ \lambda F - \K \} \cdot E -\lambda F \cdot E =  2 - (G-E)\cdot E,
\end{equation}
where we have made use of the adjunction formula \eqref{jim} once more.  Since $F \cdot E \leq 0$ and $\lambda$ is a candidate jumping number for $E$, the left side of equation \eqref{bob} is non-negative.  Hence, we must have that $(G-E)\cdot E \leq 2$.  As $G$ is connected, in fact, $(G-E)\cdot E$ is either 1 or 2, so $G$ is in fact a chain.  If $E$ is an end of $G$ so that $(G-E)\cdot E = 1$ and $E$ does not correspond to a Rees valuation of $\a$, then
\[
\{ \lambda F - \K \} \cdot E = 1.
\]
It follows that $E$ must intersect at least two components of $F$ or $\K$ which do not have $\lambda$ as a candidate jumping number.  As it also intersects a component of $G$, all of which have $\lambda$ as a candidate jumping number, the desired conclusion follows.  On the other hand, if $E$ is not an end of $G$ so that $(G-E)\cdot E = 2$, we have
\[
\{ \lambda F - \K \} \cdot E - \lambda F \cdot E = 0.
\]
Thus, both terms on the left must vanish.  In particular, $F \cdot E = 0$ so $E$ does not correspond to a Rees valuation of $E$, and $E$ can only intersect those components of $F$ which also have $
\lambda$ as a candidate jumping number.
\end{proof}

\begin{remark}
Recently, Schwede and Takagi \cite{Karl} have made use of  \emph{multiplier submodules}\footnote{These objects were also called \emph{adjoint modules} in \cite{HyrySmith03}.} in studying rational singularities of pairs. The multiplier submodules $\J(\omega_X, \a^\lambda) = \pi_* \O_Y(\lceil K_Y-\lambda F \rceil )$ are indexed by the positive rational numbers, and form nested sequence of submodules of the canonical module $\omega_X$.  These behave in a manner similar to multiplier ideals, and one can use them to define the rational threshold and rational jumping numbers of the pair $(X,\a)$.   Since multiplier submodules satisfy the analogue of local vanishing, the same methods used above apply and give similar results for critical contribution of rational jumping numbers.
\end{remark}
 
 \section{Jumping Number Algorithm and Computations}
 \label{alg}

We now describe an algorithm for computing the jumping numbers of $(X,\a)$ from a log resolution $\pi: Y \to X$.  Let $F$  be the effective divisor $F$ on $Y$ such that $\a \O_Y  = \O_Y(-F)$, and $E_1, \ldots, E_r$ the prime divisors appearing in $\K$ or $F$.

\smallskip

\noindent
{\bfseries Step 1.} Compute the coefficients  of the divisors $\K = \sum_{i=1}^r b_i E_i$ and $F = \sum_{i=1}^r a_i E_i$, and use these to find the nontrivial candidate jumping numbers $\{\, \frac{b_i + m}{a_i} \, | \, m  \in \Z_{>0} \, \}$ for each $E_i$ which are at most equal to two.

\smallskip

\noindent
{\bfseries Step 2.} Next, we must determine those $E_i$ which correspond to Rees valuations of $\a$. The $E_i$ which are not exceptional, \mbox{i.e.} the strict transforms of  divisorial components of the subscheme defined by $\a$, always correspond to Rees valuations.  The prime exceptional divisors $E_i$ corresponding to Rees valuations are characterized by the property that $E_i \cdot F < 0$.  Also determine which $E_i$ intersect at least three other $E_j$, for $j \neq i$.

\smallskip

\noindent
{\bfseries Step 3.} For each candidate jumping number $\lambda \leq 2$ appearing in the first step, perform the following series of checks to determine if $\lambda$ is realized as an actual jumping number.
\begin{enumerate}[(i)]
\item 
If $\lambda$ is a candidate jumping number for an $E_i$ which is not exceptional and corresponds to 
a Rees valuation of $\a$, then $\lambda$ is realized as a jumping number contributed by $E_i$.  Proceed to check the next candidate jumping number.  Otherwise, continue to (ii).
\item
Using the necessary geometric conditions from Theorem \ref{cor},
determine all of the connected chains of prime exceptional divisors which may critically contribute $\lambda$.  Specifically, these are the connected exceptional chains $G = E_{i_1} + \cdots E_{i_k}$ such that
$\lambda$ is a candidate jumping number for each $E_{i_j}$, and
the ends of $G$ either correspond to Rees valuations or intersect at least three other $E_i$.
\item
For each chain $G$ from (ii), use the numerical criteria of Theorem \ref{lemma} to determine if $\lambda$ is realized as a jumping number critically contributed by $G$.  Specifically,
\begin{itemize}
\item
If $G=E_{i_1}$ is prime, then $\lambda$ is (critically) contributed by $E_{i_1}$ to $X$
if and only if
\[
\lceil \K - \lambda F \rceil \cdot E_{i_1}\geq - E_{i_1}\cdot E_{i_1}.\]
\item
If $G$ is reducible, then
$\lambda$ is critically contributed by $G$ if and only if 
\[ \lceil \K - \lambda F \rceil \cdot E_{i_j} = -G\cdot E_{i_j} \]
 for each of the prime divisors $E_{i_j}$ in the support of $G$.
\end{itemize}
\item
If we are not in the situation of (i) and $\lambda$ were realized as a jumping number, it would be critically contributed by some collection of exceptional divisors.  Indeed, the sum of all the exceptional divisors in $F$ which share this candidate jumping number would contribute, and a minimal contributing collection would critically contribute.  Thus, if (i) and (iii) have produced only negative answers, we deduce that $\lambda$ cannot be a jumping number.
\end{enumerate}

\smallskip

\noindent
{\bfseries Step 4.} From above, we now know all of the jumping numbers of $(X, \a)$ which are at most two.  To determine the remaining jumping numbers, recall that the jumping numbers are eventually periodic;  $\lambda > 2$ is a jumping number if and only if $\lambda - 1$ is also a jumping number.  This concludes the algorithm for computing the jumping numbers of $(X, \a)$. 

\smallskip

The remainder of this section focuses on a general scenario to which this method applies.  We begin by altering our notation slightly.  Assume $R$ is the local ring at a rational singularity of a complex surface which is not smooth. Let $\pi : Y \to X = \spec(R)$ be the minimal resolution of singularities of $X$, and $\mathfrak{m}$ the maximal ideal of $R$.   Since $\pi$ is a composition of closed point blow-ups, and $X$ is singular, it must begin with a blow-up along this singular point.  Thus, $\pi$ is also a minimal log resolution of $\mathfrak{m}$.  In this case, the effective divisor $Z$ cut out by the principal ideal sheaf $\m \O_Y$ is called the \emph{fundamental cycle} of $X$.

\smallskip

The fundamental cycle of $X$ was first introduced by Artin in \cite{Artin1}, where it was characterized  numerically.  We now recover this description while reproducing a summary from \cite{Lipman} of results found in \cite{LipRat}. Recall that a divisor $D$ on $Y$ is said to be \emph{antinef} if $D \cdot E \leq 0$ for all prime exceptional divisors $E$ on $Y$.  By a fundamental result of Lipman, Theorem 12.1 in \cite{LipRat}, a divisor $D$ on $Y$ is antinef if and only if $\O_Y(-D)$ is globally generated.  In particular, an antinef divisor is effective.  It follows immediately that there is a bijective correspondence between complete ideals $I \subseteq R$ such that $I\O_Y$ is invertible, and antinef divisors $D$ on $Y$.  Given a complete ideal $I \subseteq R$, the principal ideal sheaf $I\O_Y$ cuts out an antinef divisor $D$.  In other words, we have that $I\O_Y = \O_Y(-D)$ where $D$ is antinef.
Conversely, if $D$ on $Y$ is antinef, then $\pi_* \O_Y(-D) = H^0(Y, \O_Y(-D))$ is a complete ideal of $R$.  This correspondence is inclusion reversing, \mbox{i.e.} larger antinef divisors correspond to smaller ideals, and $\m$-primary or \emph{finite colength} ideals correspond to exceptionally supported antinef divisors.  Since $\m$ is the largest finite colength ideal of $R$, $Z$ is the unique smallest exceptionally supported antinef divisor on $Y$.  In \cite{Artin1}, it is shown that $-Z\cdot Z$ is the multiplicity of $R$, and $- Z \cdot Z + 1$ is its embedding dimension.  To compute $Z$, one may proceed as follows.  Start with the reduced sum of all of the prime exceptional divisors on $Y$.  Add an additional prime exceptional divisor $E$ only if the intersection of $E$ with this sum is positive, and repeat this process with the new sum of exceptional divisors.  After finitely many iterations of this procedure, the corresponding sum will be antinef and must necessarily be equal to $Z$.

\smallskip

Once $Z$ has been found, in order to compute the jumping numbers of $\m$, we first need the relative canonical divisor $\K$.  Recall\footnote{Artin \cite{Artin1} attributes this fact to Mumford \cite{Mumford}, while Lipman \cite{LipRat} gives credit to Du Val.} that the restriction of the intersection product to the exceptional locus is negative definite.  Thus, to compute $\K$, it suffices to specify its intersection with any prime exceptional divisor $E$.  Since $E \cong \P^1$, the adjunction formula once more gives $\K \cdot E = -2 - E \cdot E$.  Using the algorithm for finding jumping numbers described above, this shows how to compute the jumping numbers of $\m$ starting from intersection matrix of the prime exceptional divisors on $Y$.

\begin{example}[Du Val Singularities]  
\label{Du Val}
In Table 1, we give the results of applying the above techniques to the various types of \emph{Du Val} singularities.  In this case, the relative canonical divisor of the minimal resolution is zero, and all of the prime exceptional divisors have self-intersection $-2$.  The dual graph corresponding to the exceptional locus is given by one of the Dynkin diagrams of type $A$, $D$, or $E$.  See Section 4.3 of \cite{Shafarevich88} for a full description. Recall that $\lambda > 2$ is a jumping number if and only if $\lambda - 1$ is also a jumping number.
\end{example}

{\renewcommand\arraystretch{2}
\begin{table}
\label{table}
\caption{Jumping Numbers  in  Du Val Singularities}
\begin{tabular}{c@{\hspace{.4in}}cc}
Type & Dual Graph & Jumping Numbers of the Maximal Ideal  \\ \hline
\raisebox{.3in}{$A_n$ ($n \geq 1$)} & 
 \includegraphics[width=2.2in]{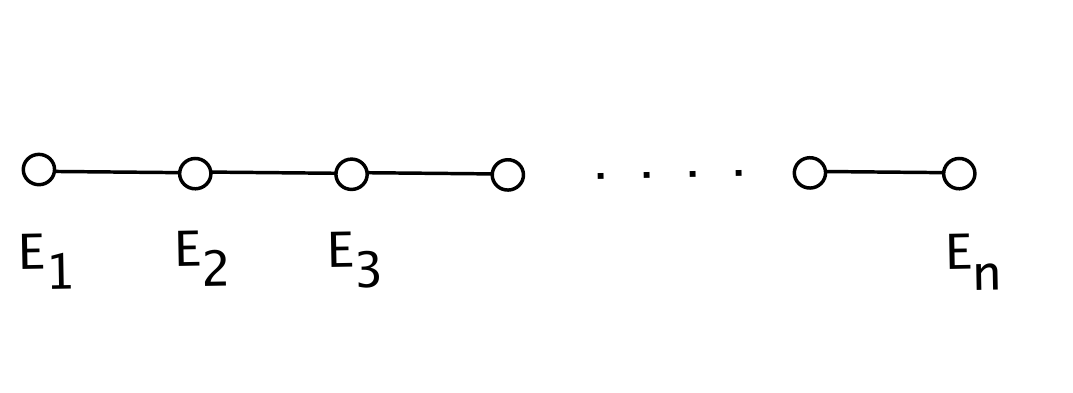}  & 
\raisebox{.3in}{$
\{ 1, 2, \ldots \}$}\\
\hline
\raisebox{.3in}{$D_n$ ($n \geq 4$)} &  \raisebox{1in}{} \includegraphics[width=2.2in]{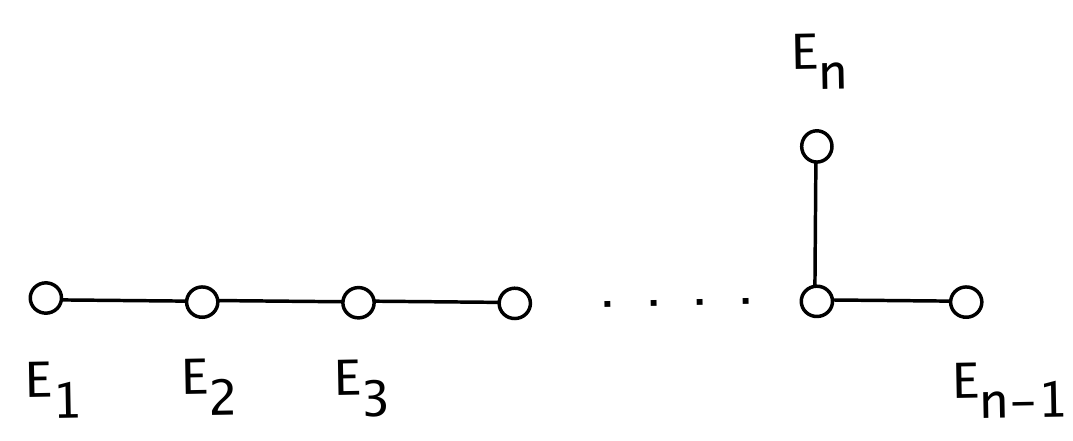} & \raisebox{.3in}{$\{ \frac{1}{2}, \frac{3}{2}, 2,  \ldots \}$}\\
\hline \smallskip
\raisebox{.3in}{$E_6$} &  \includegraphics[width=2.2in]{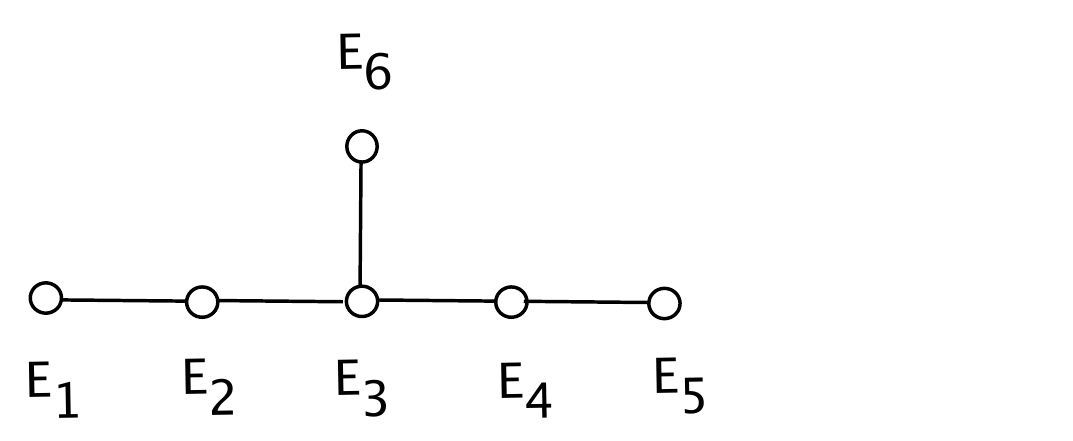} & \raisebox{.3in}{$\{ \frac{1}{3}, \frac{4}{3}, \frac{3}{2}, 2, \ldots \}$}\\
\hline \smallskip
\raisebox{.3in}{$E_7$} &  \includegraphics[width=2.2in]{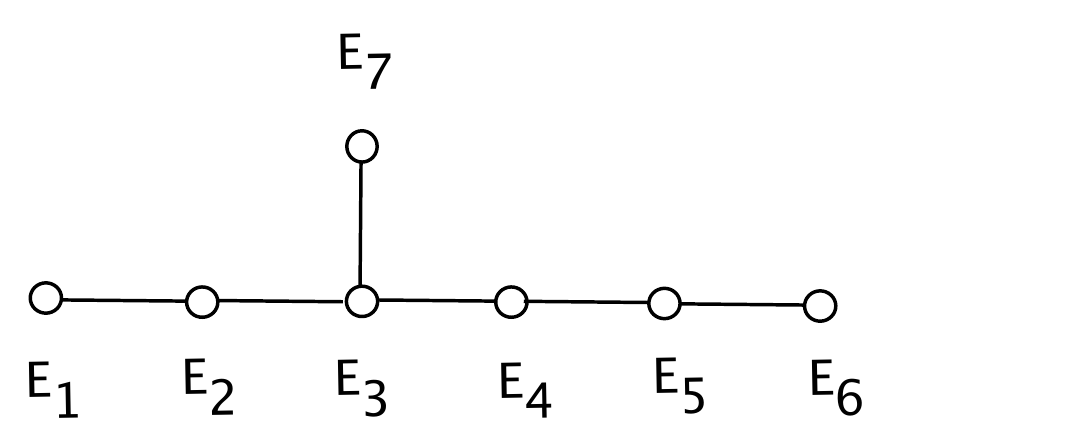} & \raisebox{.3in}{$\{ \frac{1}{4}, \frac{5}{4}, \frac{3}{2}, 2,  \ldots \}$}\\
\hline \smallskip
\raisebox{.3in}{$E_8$} &  \includegraphics[width=2.2in]{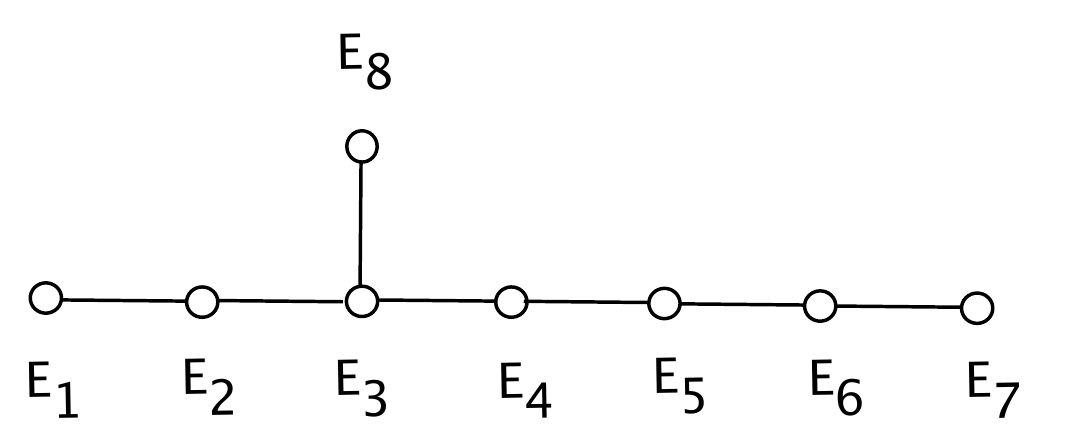} & \raisebox{.3in}{$\{ \frac{1}{6}, \frac{7}{6}, \frac{3}{2}, 2,  \ldots \}$}\\
\hline \smallskip
\end{tabular}
\end{table}}

\begin{tabular}{c@{\quad}l}
\raisebox{5ex}{\parbox{.85in}{$A_n$ ($n \geq 1$)}} & \parbox{.8\textwidth}{
The fundamental cycle is $Z = E_1 + \cdots + E_n$, and both $E_1$ and $E_n$ are Rees valuations of the maximal ideal.  The log canonical threshold $1$ is critically contributed by $E_1 + \cdots + E_n$, while all of the other jumping numbers are contributed by either $E_1$ or $E_n$.} \bigskip \\ 
\end{tabular}

\begin{tabular}{c@{\quad}l}
\raisebox{5ex}{\parbox{.85in}{$D_n$ ($n \geq 4$)}}&\parbox{.8\textwidth}{
The fundamental cycle is $Z = E_1 + E_n + E_{n-1} + 2E_2 + \cdots + 2E_{n-2}$, and $E_2$ is the only Rees valuation of the maximal ideal.  The log canonical threshold $\frac{1}{2}$ is critically contributed by $E_2 + \cdots + E_{n-2}$, while all  other jumping numbers are contributed by $E_2$.} \bigskip \\
\end{tabular}

\begin{tabular}{c@{\quad}l}
\raisebox{5ex}{\parbox{.85in}{$E_6$ \hspace{.2in}}}&\parbox{.8\textwidth}{
The fundamental cycle is $Z = E_1 + 2E_2 + 3E_3 + 2 E_4 + E_5 + 2 E_6$, and $E_6$ is the only Rees valuation of the maximal ideal.  The jumping numbers $\{ \frac{1}{3} + \Z_{\geq 0}\}$ are contributed by $E_3$,  while all  other jumping numbers $\{ \frac{3}{2} + \frac{1}{2}\Z_{\geq 0} \}$ are contributed by $E_6$.} \bigskip \\
\end{tabular}

\begin{tabular}{c@{\quad}l}
\raisebox{5ex}{\parbox{.85in}{$E_7$ \hspace{.2in}}}&\parbox{.8\textwidth}{
The fundamental cycle is 
$Z = 2 E_1 + 3E_2 + 4E_3 + 3 E_4 + 2 E_5 +  E_6 + 2E_7$, and $E_1$ is the only Rees valuation of the maximal ideal.  The jumping numbers $\{ \frac{1}{4} + \Z_{\geq 0}\}$ are contributed by $E_3$,  while all  other jumping numbers $\{ \frac{3}{2} + \frac{1}{2}\Z_{\geq 0} \}$ are contributed by $E_1$.} \bigskip \\
\end{tabular}

\begin{tabular}{c@{\quad}l}
\raisebox{5ex}{\parbox{.85in}{$E_8$ \hspace{.2in}}}&\parbox{.8\textwidth}{
The fundamental cycle is $Z = 2E_1 + 4E_2 + 6E_3 + 5 E_4 + 4 E_5 + 3 E_6 + 2E_7 + 3E_8$, and $E_7$ is the only Rees valuation of the maximal ideal.  The jumping numbers $\{ \frac{1}{6} + \Z_{\geq 0}\}$ are contributed by $E_3$,  while all  other jumping numbers $\{ \frac{3}{2} + \frac{1}{2}\Z_{\geq 0} \}$ are contributed by $E_7$.} \\
\end{tabular}

\begin{example}[Cyclic Quotient Surface Singularities]
\label{Toric}
Consider the action of the cyclic group of order $n$ on $\A^2=\spec \C[x,y]$ given by
\[
x \mapsto \zeta_n x \qquad y \mapsto (\zeta_n)^k y
\]
where $\zeta_n$ is a primitive $n$-th root of unity, and $n>k$ are relatively prime positive integers.
The quotient is a toric surface with a rational singularity.  See \cite{Fulton}, Section 2.6, for a complete discription.  Let $R$ be the local ring at the singular point, and set $X = \spec(R)$.  Consider the Hirzebruch-Jung continued fraction
\[
\frac{n}{k} = a_1 - \frac{1}{a_2 - \frac{1}{\ldots - \frac{1}{a_m}}}
\]
of $\frac{n}{k}$, with integers $a_1, \ldots, a_m \geq 2$.  The exceptional set of the minimal resolution $\pi:Y \mapsto X$ is a chain of $m$ rational curves
\begin{center}
\includegraphics[width=3in]{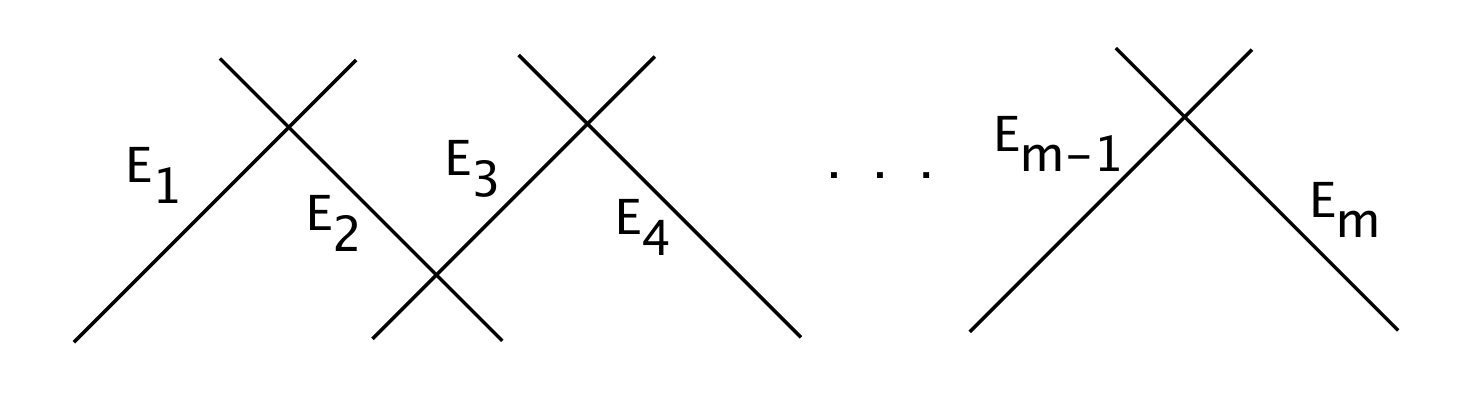}
\end{center}
where $E_i \cdot E_i = -a_i$ for $i = 1, \ldots, m$.  The fundamental cycle is $Z = E_1 + \cdots + E_m$.  To find the candidate jumping numbers of $E_i$,  set $j_0 = 1$ and $j_1 = \frac{k+1}{n}$.  Define $j_2, \ldots, j_m$ recursively by
\[
j_{i+1} = a_i j_i - j_{i-1}.
\]
One can check the nontrivial candidate jumping numbers of $E_i$ are precisely $\{ j_i + \Z_{\geq 0} \}$.  Using that each $a_i \geq 2$ and the recursive definition, it is easy to see there is some $1 \leq k_1 \leq k_2 \leq m$ such that we have the inequalities
\[
j_1 > j_2 > \cdots > j_{k_1} \qquad j_{k_1} = j_{k_1 + 1} = \cdots = j_{k_2} \qquad j_{k_2}< j_{k_2+1} < \cdots < j_m
\]
and $j_1 , j_m \leq 1$.  These relationships allow one to progressively check the numerical conditions given in Theorem \ref{lemma}, and we find the jumping numbers of the maximal ideal are precisely
\[
\min \{j_1, \ldots, j_m \} \cup \left( \bigcup_{i \in \mathscr{R} }
\{j_i + \Z_{>0} \} \right).
\]
where $\mathscr{R} = \{1, m\} \cup \{\,i  : \,\, a_i \geq 3 \,\,\}$ is the set of indices of the $E_i$ corresponding to Rees valuations of the maximal ideal.  The log canonical threshold $\min \{j_1, \ldots, j_m \}$ is critically contributed by $E_{j_{k_1} } + \cdots E_{j_{k_2}}$, while the jumping numbers $\{j_i + \Z_{>0} \}$ for $i \in \mathscr{R}$ are contributed by $E_i$.
\end{example}

 \section{Applications to Smooth Surfaces}
 \label{final}

Suppose $R$ is the local ring at a point on a smooth surface, and $\pi: Y \to X = \spec(R)$ is the minimal resolution of the divisor $C$ on $X$.  In \cite{How}, it was shown that an exceptional divisor $E$ which intersects three other prime divisors in the support of $\pi^*C$ contributes a jumping number less than one to the pair $(X,C)$. 
However, as the next example shows,
a chain of exceptional divisors $G$ in the minimal resolution of a plane curve $C$, where  the ends $E$ of $G$ intersect at least three other prime divisors in the support of $\pi^*C$, may or may not critically contribute to the jumping numbers of the embedded curve.  It remains unclear if additional geometric information would  guarantee that $G$ contributes a jumping number to $(X,C)$. A similar situation is found in \cite{Veys}, where Proeyen and Veys are concerned with the poles of the topological zeta function.  To determine whether or not a candidate pole is a pole, they also rely on both geometric and numerical data.

\begin{example}
 Suppose $C$ is germ of the plane curve defined by the polynomial $(y^2-x^5)(y^2 - x^3) = 0$.  It takes two blow-ups to separate the two components of $C$, creating divisors $E_1$ and $E_2$.  At this point these components are both smooth.  To ensure normal crossings, one must blow-up an additional point on the transform of the first component, and two additional points on the second, creating divisors $E_3$, $E_4$, and $E_5$, respectively.   One checks
 \[
 \pi^*C = C+ 4E_1 + 7E_2 +12E_3 +8E_4 + 16E_5 \qquad
 \K = E_1 +2E_2+4E_3+3E_4+6E_5
 \]

\smallskip 

 \begin{center}
 \includegraphics[width=1.3in]{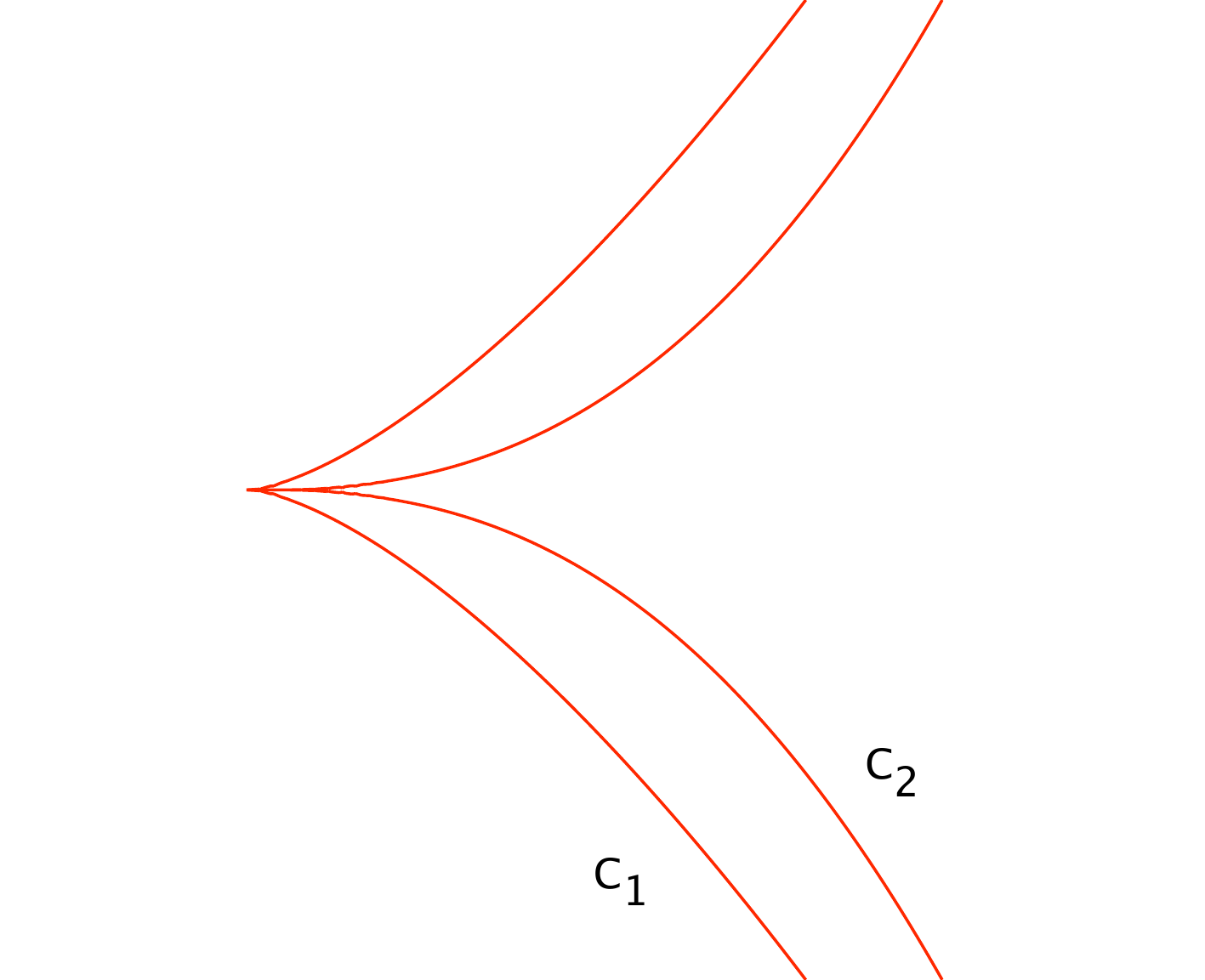}
 \raisebox{.55in}{$\arrow{w}$} \hspace{-.1in}
  \includegraphics[width=1.3in]{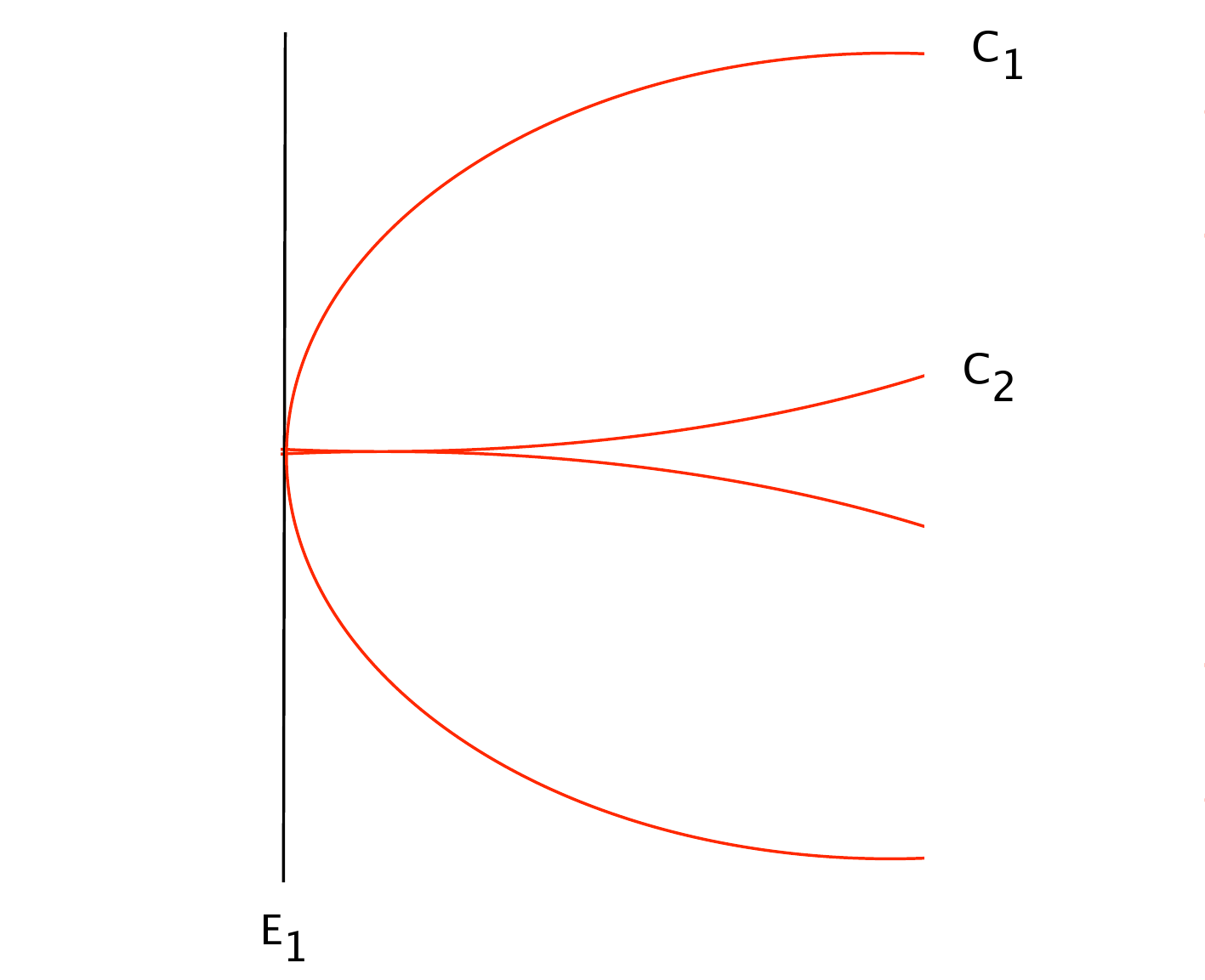}
 \hspace{.1in}
 \raisebox{.55in}{$\arrow{w}$} \hspace{-.1in}
  \includegraphics[width=1.3in]{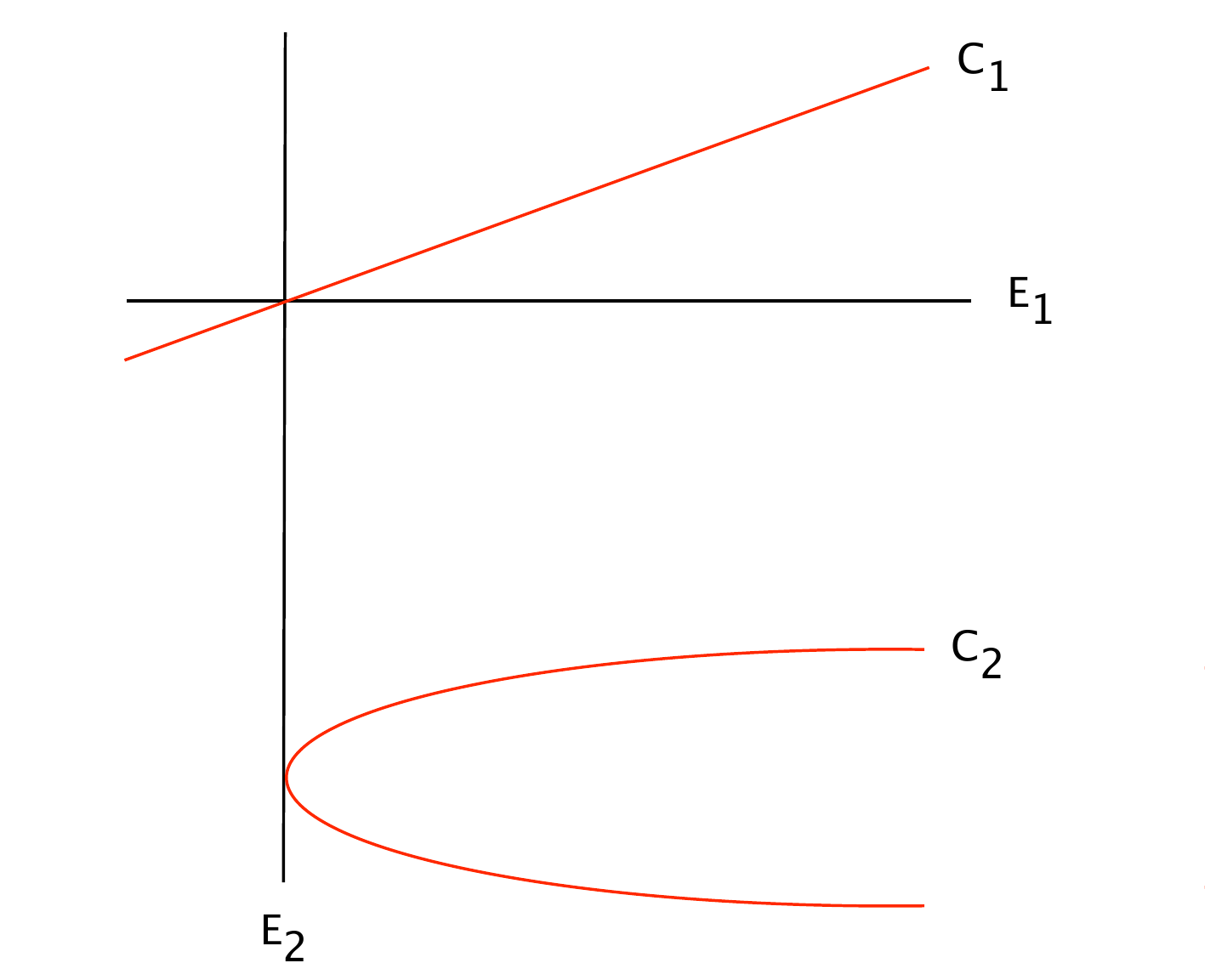}
 \hspace{.1in}
 \raisebox{.55in}{$\arrow{w}$} \hspace{-.1in}
  \includegraphics[width=1.3in]{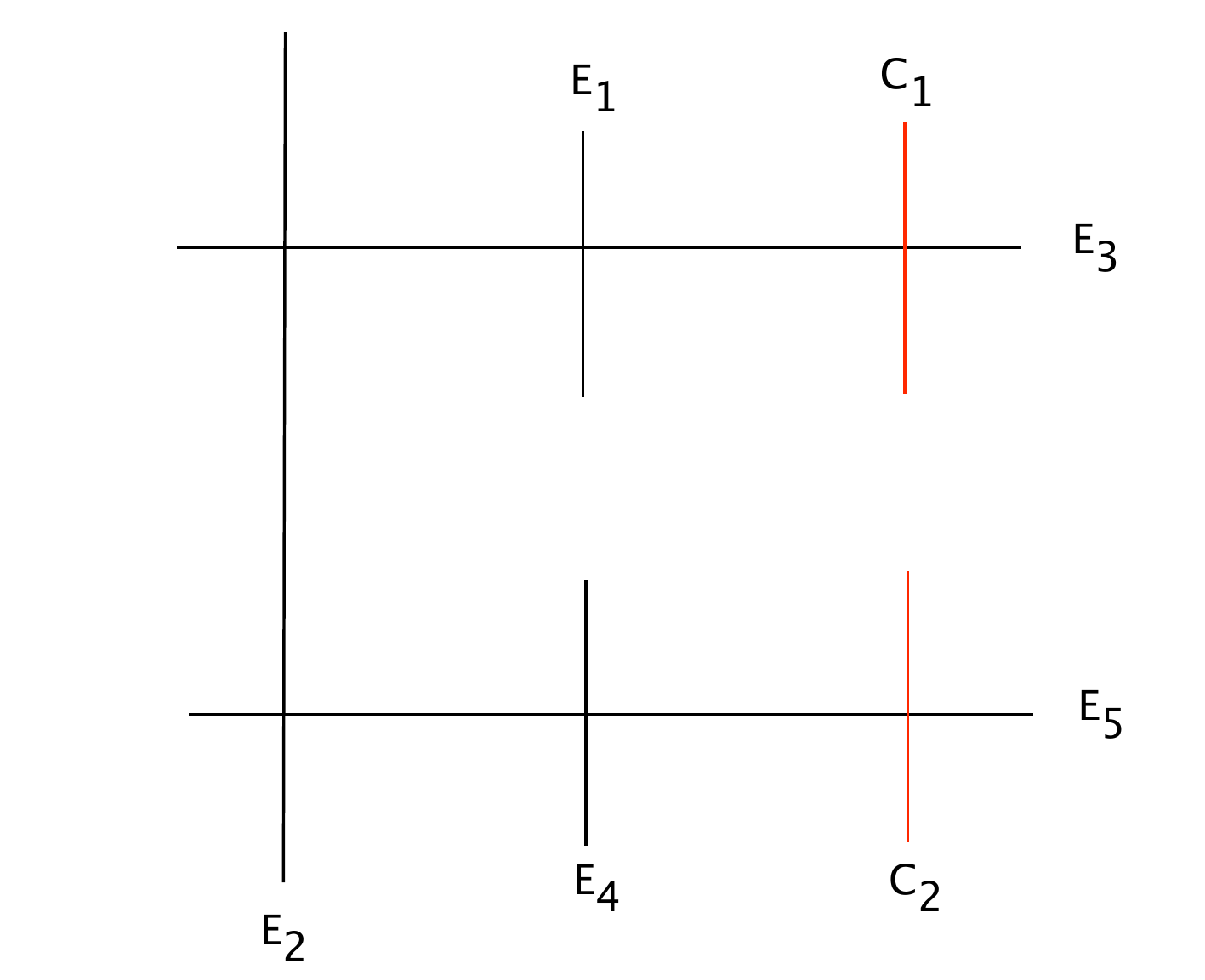}
  \end{center}

\smallskip

\noindent By the Theorem \ref{cor}, the only possible chain of length greater than one that can contribute a jumping number is $E_2 + E_3 + E_5$.  However, these three divisors do not share a common candidate jumping number less than one; hence, they cannot critically contribute any jumping number less than one.  Notice the similarity between the exceptional divisors here and those in Example~\ref{secondexample}.  Despite the fact that the corresponding chains ($E_2 + E_3 + E_5$ here, and $E_0 + E_2 + E_2'$ in Example~\ref{secondexample}) intersect their complements the same number of times, one chain contributes a jumping number while the other does not.  
\end{example}

\begin{remark}
One of the initial motivations for studying jumping number contribution was to further understand the work of J\"arvilehto \cite{Tarmo}. It gives an explicit formula for the jumping numbers of the germ of an analytically irreducible plane curve.  In addition, the jumping numbers are partitioned into subsets corresponding to prime divisors on the minimal resolution which intersect three others in the pull-back of the curve.  In future work, we will show the jumping numbers less than one are always contributed by prime exceptional divisors, and  the partitioning in \cite{Tarmo} corresponds to grouping together the jumping numbers contributed by a particular prime divisor.

\smallskip

J\"arvilehto simultaneously computes the jumping numbers of any simple complete finite colength ideal in a two dimensional regular local ring, and in doing so shows that 1 is not a jumping number of such an ideal.  Our next result  uses Theorem \ref{cor} to give an alternative proof of this fact, and concludes further that this gives a criterion for simplicity.
\end{remark}

For the remainder of this section, we fix the following notation.  Let $R$ be the local ring at a point on a smooth complex surface, and $\pi:Y \to X = \spec(R)$ the minimal resolution of a complete finite colength ideal $\a \subseteq R$ such that $\a\O_Y = \O_Y(-F)$.  Note that the numerical criterion for critical contribution can be simplified using the adjunction formula.  A single exceptional prime divisor $E$ contributes a candidate jumping number $\lambda$ if and only if $-\lfloor \lambda F \rfloor \cdot E \geq 2$;  a reducible chain of exceptional divisors $G$ with common candidate jumping number $\lambda$ critically contributes $\lambda$ if and only if the ends $E$ of $G$ satisfy $-\lfloor \lambda F \rfloor \cdot E = 1$, and the non-ends $E'$ of $G$ satisfy $-\lfloor \lambda F \rfloor \cdot E' = 0$.
 
 Before we begin, it is first necessary to review some of the Zariski-Lipman theory of complete ideals in two dimensional regular local rings.  A good summary of this theory can be found in the introduction to \cite{Lipman}, as well as \cite{Tarmo}.
 Let $E_1, \ldots, E_n$ be the prime exceptional divisors, and consider $\Lambda = \Z E_1 + \cdots \Z E_n$ the lattice they generate. There is another free $\Z$-basis for $\Lambda$, denoted $\hat{E_1}, \ldots, \hat{E_n}$, defined by the property that $\hat{E_i}\cdot E_i =-1$ and $\hat{E_i}\cdot E_j = 0$ for $i \neq j$.  Note that these divisors generate the semigroup of antinef divisors in $\Lambda$.  Indeed, $D = \hat{d_1}\hat{E_1} + \cdots + \hat{d_n}\hat{E_n}$ is antinef if and only if $\hat{d_i} = - D \cdot E_i \geq 0$ for all $i = 1, \ldots, n$.  It is not hard to see that the corresponding complete finite colength ideals $P_i = \pi_*\O_Y(-\hat{E_i})$ are simple, i.e. cannot be written nontrivially as a product of ideals.

 Suppose $I = \pi_*\O_Y(-D)$ is the complete finite colength ideal corresponding to the antinef divisor $D = \hat{d_1}\hat{E_1} + \cdots + \hat{d_n}\hat{E_n} \in \Lambda$.
Then we see immediately
$I = P_1^{\hat{d_1}}\cdots P_n^{\hat{d_n}}$, and this factorization is unique as $\hat{E_1}, \ldots, \hat{E_n}$ are a basis for $\Lambda$.  Further, the valuations on $\mathcal{K}$ corresponding to those $E_i$ such that $\hat{d_i}$ are nonzero are precisely the Rees valuations of $I$.   As any complete ideal 
can be written uniquely as the product of a principal ideal and a finite colength complete ideal,\begin{footnote}{Observe that $I = (I^{-1})^{-1} \cdot (II^{-1})$ shows how this is achieved.}\end{footnote}  unique factorization extends to all complete ideals of $R$.

\begin{proposition}
\label{app}
A complete finite colength ideal $\a$ in the local ring of a smooth complex surface is  simple if and only if $1$ is not a jumping number of $(X,\a)$.
\end{proposition}

\begin{proof}
If $\a$ is simple, then $\a = P_i$ for some $i$, and we have that $\a \O_Y = \O_Y(-\hat{E_i})$.  Suppose, by way of contradiction, $1$ is a jumping number of $(X,\a)$.  We may assume there is a reduced chain   of exceptional divisors $G$ which critically contributes $1$ to the pair $(X,\a)$.  $G$ cannot be a single prime divisor $E$ since $-\hat{E_i}\cdot E$ is either 0 or 1. Thus, $G$ is reducible and must have two distinct ends satisfying $-\hat{E_i}\cdot E = 1$.  Since this only happens for $E = E_i$,  $1$ is not a jumping number of $(X,\a)$.

Alternatively, assume that $\a = P_1^{\hat{d_1}}\cdots P_n^{\hat{d_n}}$ is the finite colength complete ideal corresponding to the antinef divisor $D = \hat{d_1}\hat{E_1} + \cdots + \hat{d_n}\hat{E_n}$, and $\a$ is not simple.  Suppose first there is some $i$ such that $\hat{d_i}\geq 2$.  In this case, $-D \cdot E_i = \hat{d_i} \geq 2$ shows that $1$ is a jumping number contributed by $E_i$.  Otherwise, we may assume $\hat{d_i}$ is $0$ or $1$ for each $i$, and at least two such are nonzero.  In this case, we can find two of them $\hat{d_{i_1}} = \hat{d_{i_2}}=1$ such that for any $E_j$ in the unique chain of exceptional divisors $G$ connecting $E_{i_1}$ and $E_{i_2}$ we have $\hat{d_j} = 0$.  Theorem \ref{cor} now gives that $1$ is a jumping number of $(X,\a)$ critically contributed by~$G$.
\end{proof}

\begin{remark}
The technique used in Corollary \ref{app} also shows that every chain of exceptional divisors critically contributes a jumping number for some ideal $\a \subset R$ having $\pi$ as a resolution.  Indeed, if $G$ is the chain connecting $E_{i_1}$ and $E_{i_2}$, then $G$ critically contributes $1$ to the ideal $P_{i_1} P_{i_2}$.  
One can also use this method to produce examples where many intersecting and nonintersecting chains  critically contribute the same jumping number to a given pair.
\end{remark}

\providecommand{\bysame}{\leavevmode\hbox to3em{\hrulefill}\thinspace}
\providecommand{\MR}{\relax\ifhmode\unskip\space\fi MR }
\providecommand{\MRhref}[2]{%
  \href{http://www.ams.org/mathscinet-getitem?mr=#1}{#2}
}
\providecommand{\href}[2]{#2}

\end{document}